\documentclass[preprint,12pt]{elsarticle}

\usepackage{amsfonts}
\usepackage{amssymb}
\usepackage{mathtools}
\usepackage{tikz}
\usetikzlibrary{arrows.meta, positioning, calc}
\usepackage{tikz-cd}
\usepackage{booktabs}
\usepackage{array}
\usepackage{enumitem}
\setlist[enumerate]{leftmargin=.5in}
\setlist[itemize]{leftmargin=.5in}
\usepackage{url}

\graphicspath{{./}}

\usepackage[capitalise,nameinlink]{cleveref}

\usepackage{amsthm}
\newtheorem{theorem}{Theorem}[section]

\newtheorem{corollary}[theorem]{Corollary}
\newtheorem{proposition}[theorem]{Proposition}
\theoremstyle{definition}
\newtheorem{definition}[theorem]{Definition}
\newtheorem{example}[theorem]{Example}
\theoremstyle{remark}
\newtheorem{remark}[theorem]{Remark}

\newcommand{\kk}{k}
\newcommand{\QQ}{\mathbb{Q}}
\newcommand{\ZZ}{\mathbb{Z}}
\newcommand{\RR}{\mathbb{R}}
\newcommand{\CC}{\mathbb{C}}

\newcommand{\Var}{\mathbf{V}}

\newcommand{\GB}{\mathrm{GB}}

\newcommand{\conv}{\mathrm{conv}}

\DeclareMathOperator{\Fam}{Fam}

\DeclareMathOperator{\Pairs}{Pairs}
\DeclareMathOperator{\rank}{rank}

\journal{Journal of Computational Algebra}

\begin{document}

\begin{frontmatter}

\title{Algebraic Varieties and Ideal Theory\\
  in Combinatorial Click-Reaction Design}

\author[upc]{Vicent Ribas Ripoll\corref{cor1}}
\ead{vribas@ieee.org}
\cortext[cor1]{Corresponding author}
\address[upc]{Barcelona, Spain}

\begin{abstract}
In this short paper we study compatibility-constrained combinatorial
chemical assembly problems through the lens of commutative algebra.
Given a finite set~$F$ of \emph{chemical families}, a finite
set~$H$ of \emph{handle types}, and
a compatibility relation $\Pairs(f) \subseteq H \times H$ for each
$f \in F$, we construct an \emph{assembly ideal}
$I = J_{\mathrm{bool}} + J_{\mathrm{sel}} + K_{\mathrm{compat}}$
in a polynomial ring $R = \kk[F, H, H']$ whose variety
$\Var(I) \subseteq \{0,1\}^n$ is the set of feasible triples.
We prove that $I$ is zero-dimensional and radical, so
$R/I \cong \kk^{|\Var(I)|}$.  Elimination ideals characterise
\emph{handle diagnosticity} (whether a handle determines its
family), the toric ideal of the log-linear model on~$\Var(I)$
measures the redundancy of the compatibility relation, and a
multi-step ideal~$I^{(k)}$ encodes orthogonality constraints
between simultaneous assembly plans; the clique number
$\omega(G_\perp)$ of the associated orthogonality graph gives
the maximum number of plans that can coexist without
cross-reactivity.  We derive a necessary and sufficient criterion
for a new family to raise~$\omega$.
The framework is instantiated on the bioorthogonal click chemistry
landscape ($|F| = 8$, $|H| = 17$), yielding
$|\Var(I)| = 30$, a toric ideal with 2~generators,
$\mathrm{ML\,degree} = 1$, and $\omega(G_\perp) = 4$.
All computations are verified over~$\QQ$ in SymPy.
\end{abstract}


\end{frontmatter}

\section{Introduction}\label{sec:intro}

\paragraph{Motivation: designing an ADC by dual click chemistry.}
Antibody--drug conjugates (ADCs) are therapeutic molecules in
which a monoclonal antibody is covalently linked to a cytotoxic
payload, delivering the drug selectively to cells displaying the
antigen recognised by the antibody.  Kadcyla (trastuzumab
emtansine), approved for HER2-positive metastatic breast cancer,
is a canonical example: its antibody trastuzumab is linked to the
maytansinoid DM1 through a maleimide--thioether chemistry installed
on lysine residues~\cite{Lang14}.  Maleimide chemistry is simple
but not bioorthogonal and can exchange with free thiols in plasma,
which motivates the use of \emph{bioorthogonal click
reactions}~\cite{Sletten09,Devaraj11}---highly selective ligations
that proceed under mild, aqueous conditions without interfering
with native biological functionality.

A promising design strategy is a \emph{dual-click assembly}:
install a first click handle on the antibody, then attach the
payload in two orthogonal click steps via a bifunctional linker
carrying one complementary handle at each end.  For concreteness,
install a cyclopropene on the antibody (e.g.\ by genetic code
expansion at an engineered residue); the bifunctional linker
bears a tetrazine at one end and an azide at the other; the
payload DM1 bears a dibenzocyclooctyne (DBCO).  Step~1 attaches
the linker to the antibody by inverse electron-demand Diels--Alder
(IEDDA) between cyclopropene and tetrazine; Step~2 attaches DM1
to the linker by strain-promoted azide--alkyne cycloaddition
(SPAAC) between azide and DBCO.  The resulting conjugate is shown
in \cref{fig:kadcyla}; this is the concrete molecule we return to
throughout the paper to interpret each algebraic invariant we
compute.  The design must guarantee that
(i) the two reactions are compatible with their handle partners,
and (ii) the cyclopropene, tetrazine, azide and DBCO handles do
not cross-react within or between steps.  This is a small instance
of a larger combinatorial question: which triples of family,
source handle, and target handle are feasible, and which tuples of
such triples can be run in sequence or in parallel without
cross-reactivity?

\begin{figure}[ht]
\centering
\includegraphics[width=0.98\textwidth]{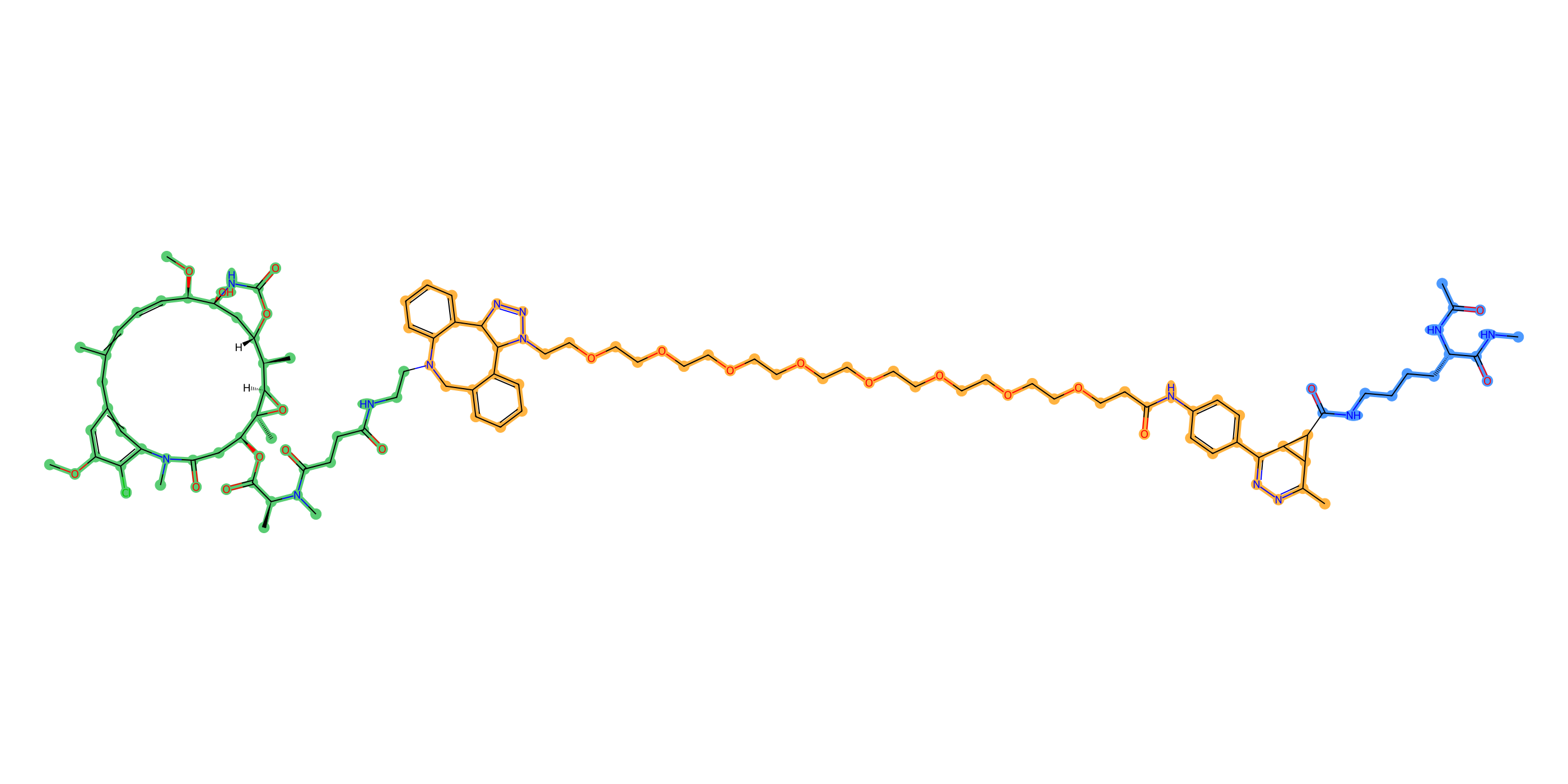}
\caption{The working conjugate of this paper: the dual-click
  redesign of Kadcyla formalised below in \cref{ex:kadcyla}.
  The maytansinoid payload DM1 (left, green) is joined to the
  linker through a SPAAC triazole (DBCO + azide, step~2,
  orange); a PEG spacer (centre, orange) bridges to a
  dihydropyridazine IEDDA product (cyclopropene + methyltetrazine,
  step~1, orange/blue), which is installed on the engineered
  residue of the antibody, represented here by an acetyl-lysine
  stand-in (right, blue).  The original SMCC thioether of
  Kadcyla is replaced by the two bioorthogonal click junctions;
  every subsequent algebraic result in the paper is interpreted
  on this molecule.}
\label{fig:kadcyla}
\end{figure}

Since the seminal copper-catalysed azide--alkyne cycloaddition
(CuAAC) of Sharpless and Meldal~\cite{Rostovtsev02,Tornoe02}, the
bioorthogonal toolkit has grown to eight established reaction
families---including strain-promoted variants (SPAAC), inverse
electron-demand Diels--Alder ligations (IEDDA), oxime and
hydrazone condensations, Staudinger ligations, thiol--ene
additions, and photo-click reactions~\cite{Lang14,McKay14}---and
at least seventeen reactive-group handles.  Some handles
participate in more than one family, creating cross-reactivity
constraints that limit how many reactions can run simultaneously
in the same pot.  Designing multiplexed bioconjugation protocols
therefore requires solving a compatibility-constrained
combinatorial assembly problem over families, handles, and
feasible pairings.  Enumeration can answer a specific question for
a fixed toolkit but cannot reason about the design space itself:
it does not reveal which handles are structurally diagnostic of
their family, whether the model of ``random'' feasible assembly
has a closed-form maximum-likelihood estimator, how many reactions
can at most be run orthogonally, and what conditions an emerging
ninth family would have to satisfy to raise that maximum.

\paragraph{An algebraic formulation.}
We take an algebraic view.  Abstractly, the design problem is a
\emph{family--handle pair structure}: a finite set~$F$ of
\emph{families}, a finite set~$H$ of \emph{handle types}, and for
each family~$f$ a compatibility relation
$\Pairs(f) \subseteq H \times H$.  A \emph{feasible triple}
is a tuple $(f, h, h') \in F \times H \times H$ with
$(h, h') \in \Pairs(f)$.
We encode this structure as an ideal~$I$ in a multivariate
polynomial ring over a field~$\kk$, so that the set of feasible
triples becomes the Boolean variety $\Var(I) \subseteq \{0,1\}^n$,
and reduce the design questions above to standard operations in
commutative algebra---primary decomposition, elimination ideals,
Gr\"obner bases, toric ideals, and clique numbers of derived
graphs---following the polynomial method in
combinatorics~\cite{CLO15} and the language of algebraic
statistics~\cite{PS05,DSS09}.  The Kadcyla--cyclopropene scenario
described above is a two-step feasible plan in this formalism,
and we return to it throughout the paper to interpret the
algebraic invariants we compute.

\paragraph{Why algebra?}
Three features distinguish the algebraic route from direct
enumeration or graph-theoretic ad-hoc methods.
First, the ideal $I$ is a \emph{single object} that encodes every
feasibility constraint simultaneously; adding or removing a reaction
family amounts to adding or removing generators, with all consequences
propagated automatically by Gr\"obner-basis computation.
Second, the theory supplies \emph{certificates}: a handle is
diagnostic if and only if a certain polynomial lies in an elimination
ideal, and two plans are orthogonal if and only if their joint
ideal is the unit ideal---these are machine-checkable proofs, not
case-by-case verifications.
Third, the framework is \emph{modular}: the general results
(zero-dimensionality, radicality, fibre decomposition,
elimination-based diagnosticity) hold for any family--handle pair
structure, while the numerical invariants (30 feasible triples,
$\omega = 4$, ML degree~$= 1$, two toric generators) are specific
to the 8-family bioorthogonal instance we study in detail.

\medskip
\noindent\textbf{Contributions.}
Our main results are:

\begin{enumerate}[label=(\roman*),nosep]
\item The assembly ideal $I$ is zero-dimensional and radical, so
  $R/I \cong \kk^{|\Var(I)|}$ is a product of copies of the ground
  field---one per feasible triple.  The quotient admits a layered
  filtration through products of fields, with each surjection
  corresponding to an explicit ideal quotient (\cref{sec:ring}).
\item For each handle $h \in H$, write $\Fam(h)$ for the set of
  families that use~$h$; we call $h$ \emph{diagnostic} when
  $|\Fam(h)| = 1$.  Diagnosticity is characterised algebraically
  by membership of $h \cdot (1 - f_0)$ in the elimination ideal
  $I \cap \kk[F, H]$; in the 8-family bioorthogonal instance,
  12 of 17 handles are diagnostic (\cref{sec:elimination}).
\item For the 8-family instance, the toric ideal of the log-linear
  model on~$\Var(I)$ is generated by exactly 2 binomials and the
  ML degree equals~1, so the MLE is a rational function of the data
  (\cref{sec:statistics}).
\item The multi-step ideal $I^{(k)}$ breaks the single-step fibre
  decomposition; the clique number $\omega(G_\perp)$ of the
  \emph{orthogonality graph} (whose vertices are feasible triples
  and whose edges connect pairwise compatible triples) satisfies
  $\omega(G_\perp) = 4$ for the 8-family instance, with the
  obstruction concentrated in four corridors of the
  cross-reactivity graph (\cref{sec:multistep}).
\end{enumerate}

\Cref{tab:dictionary} summarises the correspondence between
algebraic operations and their combinatorial or chemical meaning.

\begin{table}[ht]
\centering
\caption{Algebraic--chemical dictionary.}
\label{tab:dictionary}
\begin{tabular}{@{}lll@{}}
\toprule
\textbf{Algebra} & \textbf{Operation} & \textbf{Chemical meaning} \\
\midrule
$\Var(I) \subseteq \{0,1\}^n$ & Zero set & Feasible assembly triples \\
$I = \sqrt{I}$ & Radical ideal & No hidden multiplicities \\
$\GB(I)$ (lex) & Gr\"obner basis & Triangular enumeration of assemblies \\
$I \cap \kk[F, H]$ & Elimination & Handle diagnosticity \\
$I \cap \kk[H, H']$ & Elimination & Reaction graph (family-free) \\
$\dim_\kk R/I$ & Vector space dim & Count of feasible assemblies \\
$I^{(k)}$ & Multi-step ideal & Orthogonal reaction plans \\
$\omega(G_\perp)$ & Clique number & Maximum simultaneous reactions \\
\bottomrule
\end{tabular}
\end{table}

\paragraph{Related work.}
Algebraic geometry has been applied to chemical reaction networks
primarily through the deficiency theory of Feinberg~\cite{Feinberg19}
and the toric dynamical systems programme; Dickenstein~\cite{Dickenstein16}
and Feliu--Shiu~\cite{FeliuShiu25} provide surveys.  That line of
work concerns \emph{continuous
kinetics}---the steady-state ideal of a mass-action system---whereas
our setting is a \emph{discrete combinatorial} assembly problem on a
Boolean variety.  The algebraic tools we employ---radical ideals,
elimination, toric ideals, and Gr\"obner bases---are closer in
spirit to the polynomial method in combinatorics~\cite{CLO15} and to
the algebraic statistics of log-linear
models~\cite{PS05,DSS09}.  The connection between toric ideals and
integer programming is classical; see Sturmfels~\cite{Sturmfels96}.
For the chemistry of bioorthogonal click reactions we refer to
Sletten--Bertozzi~\cite{Sletten09} and
Devaraj--Weissleder~\cite{Devaraj11}.

\paragraph{Outline.}
\Cref{sec:ring} defines the assembly ring and ideal and
establishes radicality.
\Cref{sec:groebner} treats Gr\"obner bases, the shape lemma,
and the fibre decomposition.
\Cref{sec:elimination} develops elimination theory and the
diagnosticity theorem.
\Cref{sec:statistics} introduces the algebraic statistics of
$\Var(I)$, including the toric ideal and ML degree.
\Cref{sec:multistep} defines the multi-step ideal, proves
the $\omega = 4$ orthogonality bound, and analyses conditions under
which emerging reactions can raise it.
\Cref{sec:computation} reports computational results.
\Cref{sec:discussion} discusses limitations and open problems.
\Cref{sec:conclusions} summarises the algebraic contributions
and their concrete implications for the bioorthogonal click chemistry
landscape.

\section{The Assembly Ring and Ideal}\label{sec:ring}

\subsection{Setup}

We fix a ground field $\kk$ of characteristic zero.  The
assembly ideal has generators with integer coefficients
(\cref{def:ideal}), so its algebraic structure---membership,
radicality, primary decomposition, elimination---is defined
over~$\QQ$ and all exact computations in the paper are carried
out in $\kk = \QQ$.  Geometric invariants that are intrinsically
defined over an algebraically closed field---in particular the
ML degree of the log-linear extension to~$\Var(I)$, which counts
critical points of the likelihood function over~$\CC$---are
computed after base change to~$\CC$.  Because $\Var(I)$ is
zero-dimensional (\cref{prop:zero-dim} below) and contained in
$\{0,1\}^n \subset \QQ^n$, the variety is the same over any
characteristic-zero field; only $\kk$-algebra invariants that
refer to the base field, such as the decomposition
$R/I \cong \kk^{|\Var(I)|}$, depend on the choice of~$\kk$, and
they do so only up to base change.

The input data consists of:

\begin{itemize}[nosep]
\item A finite set $F = \{f_1, \ldots, f_r\}$ of \emph{families}.
\item A finite set $H = \{h_1, \ldots, h_m\}$ of \emph{handle types}.
\item For each family $f_i \in F$, a set
  $\Pairs(f_i) \subseteq H \times H$ of \emph{compatible pairs}:
  the ordered pairs of handles that $f_i$ admits.
\end{itemize}

\noindent
In the bioorthogonal click chemistry instantiation, $F$ consists
of reaction families (e.g.\ SPAAC, CuAAC, IEDDA) and $H$ of
reactive functional groups (e.g.\ azide, alkyne, tetrazine, thiol).

\begin{definition}[Assembly Polynomial Ring]\label{def:ring}
The \emph{assembly polynomial ring} is
\[
  R \;=\; \kk[\,\underbrace{f_1, \ldots, f_r}_{F},\;
              \underbrace{h_1, \ldots, h_m}_{H},\;
              \underbrace{h'_1, \ldots, h'_m}_{H'}\,]
\]
where $H'$ is a second copy of $H$.
The total number of variables is $n = r + 2m$.
\end{definition}

The variables $f_i$ are indicator variables for the choice of
family; $h_j$ and $h'_j$ are indicator variables for the
source-side and target-side handle, respectively.  In the
chemical instantiation, $h$ and $h'$ represent the two
complementary reactive groups participating in a reaction.

\begin{definition}[Assembly Ideal]\label{def:ideal}
The \emph{assembly ideal} is $I = J_{\mathrm{bool}} + J_{\mathrm{sel}} + K_{\mathrm{compat}}$
where:
\begin{align}
  J_{\mathrm{bool}} &= \langle\, x^2 - x \;:\; x \in F \cup H \cup H'\,\rangle,
    \label{eq:jbool} \\
  J_{\mathrm{sel}} &= \bigl\langle\, \textstyle\sum_{i} f_i - 1,\;
    \textstyle\sum_{j} h_j - 1,\;
    \textstyle\sum_{j} h'_j - 1 \,\bigr\rangle,
    \label{eq:jsel} \\
  K_{\mathrm{compat}} &= \langle\, f_i \cdot h_a \cdot h'_b
    \;:\; (h_a, h_b) \notin \Pairs(f_i) \,\rangle.
    \label{eq:kcompat}
\end{align}
\end{definition}

The component $J_{\mathrm{bool}}$ cuts out the Boolean hypercube
$\{0,1\}^n$; $J_{\mathrm{sel}}$ imposes the simplex constraint
(exactly one~1 per coordinate block); $K_{\mathrm{compat}}$ kills
monomials corresponding to incompatible family--handle triples.

\begin{theorem}[Feasibility Correspondence]\label{thm:feasibility}
Since $J_{\mathrm{bool}} \subset I$, the variety
$\Var(I) \subseteq \{0,1\}^n$.  There is a bijection
\[
  \Var(I) \;\longleftrightarrow\;
  \bigl\{\,(f, h, h') \in F \times H \times H :
  (h, h') \in \Pairs(f)\,\bigr\}.
\]
\end{theorem}

\begin{proof}
A point $p \in \{0,1\}^n$ lies in $\Var(J_{\mathrm{bool}} + J_{\mathrm{sel}})$
if and only if exactly one $f_i$, one $h_a$, and one $h'_b$ equal~1,
all others~0.  Such a point additionally satisfies $K_{\mathrm{compat}} = 0$
if and only if the monomial $f_i h_a h'_b$ does \emph{not} appear among
the generators of $K_{\mathrm{compat}}$, i.e.\
$(h_a, h_b) \in \Pairs(f_i)$.
\end{proof}

\paragraph{Interpretation.}
The Boolean variety $\Var(I)$ parametrises exactly the feasible
triples of the combinatorial assembly problem.
Every subsequent algebraic operation---elimination, Gr\"obner
bases, toric ideals---acts on this variety and hence on the
design space.

\paragraph{On the running example.}
Each step of the Kadcyla--cyclopropene plan
(\cref{ex:kadcyla}) is a lattice point of the Boolean variety:
step~1 is the indicator of
(IEDDA,\,cyclopropene,\,tetrazine), and step~2 is the indicator
of (SPAAC,\,azide,\,DBCO).  The existence of the ADC design thus
reduces to checking membership of two explicit points in
$\Var(I)$, which---by the correspondence above---is the same as
checking two Boolean incidences in the compatibility relation
$\Pairs$.

\subsection{Zero-dimensionality, radicality, and the layered filtration}

\begin{theorem}[Zero-dimensionality and Radicality]\label{thm:radical}
\label{prop:zero-dim}%
The assembly ideal $I$ is:
\begin{enumerate}[nosep,label=(\alph*)]
\item \textbf{Zero-dimensional:} $\Var(I) \subseteq \{0,1\}^n$ is a
  finite set of $\kk$-rational points.
\item \textbf{Radical:} $I = \sqrt{I}$.
\end{enumerate}
Consequently, the quotient $R/I$ is a reduced $\kk$-algebra and
\[
  R/I \;\cong\; \kk^{|\Var(I)|}, \qquad
  \dim_\kk R/I \;=\; |\Var(I)| \;<\; \infty.
\]
\end{theorem}

\begin{proof}
\textit{(a)}\; Since $J_{\mathrm{bool}} \subset I$, we have
$\Var(I) \subseteq \Var(J_{\mathrm{bool}}) = \{0,1\}^n$, a finite
set of $2^n$ points in $\mathbb{A}^n(\overline{\kk})$.  Hence
$\Var(I)$ is finite, i.e.\ zero-dimensional, and all its points
lie in $\kk^n$.

\textit{(b)}\; The quotient $R/J_{\mathrm{bool}}$ is isomorphic
(as a $\kk$-algebra) to $\kk^{2^n}$, a product of $2^n$ copies
of the field~$\kk$ (the Chinese Remainder Theorem;
see e.g.~\cite{Eisenbud95}, \S2.4).  Every ideal in a product of
fields is radical, being an intersection of maximal ideals (each a
kernel of a coordinate projection).  The image of~$I$ in
$R/J_{\mathrm{bool}}$ is therefore radical, and since
$J_{\mathrm{bool}}$ is itself radical (its variety $\{0,1\}^n$ is
reduced), $I$ is radical.

\textit{(Dimension count.)}\; Since $I$ is zero-dimensional and
radical, the Finiteness Theorem
(\cite{CLO15}, Chapter~5, Theorem~6) gives
$\dim_\kk R/I = |\Var(I)|$.  The decomposition
$R/I \cong \kk^{|\Var(I)|}$ then follows from the Chinese
Remainder Theorem applied to the product-of-fields structure of
$R/J_{\mathrm{bool}}$.
\end{proof}

\begin{remark}\label{rem:zero-dim}
Zero-dimensionality is meant in the ideal-theoretic sense
$\dim(R/I) = 0$ (Krull dimension).  The ambient ring~$R$ has Krull
dimension~$n$; the Boolean constraint $J_{\mathrm{bool}}$ forces
the quotient down to dimension zero.  This ensures
Gr\"obner-basis enumeration terminates
(\cref{thm:shape}) and makes the ML degree a well-defined finite
invariant (\cref{sec:mldeg}).  The quotient $R/I$ is
\emph{reduced}---that is, has no nilpotent elements---precisely
because $I$ is radical; equivalently, $\Var(I)$ carries no embedded
multiplicities.  All points are $\kk$-rational, so results stated
over $\kk = \QQ$ extend unchanged to any characteristic-zero field.
\end{remark}

\paragraph{Interpretation.}
The identity $\dim_\kk R/I = |\Var(I)|$ means the $\kk$-vector-space
dimension of the quotient ring counts feasible triples exactly.
Membership testing in~$I$ provides a definitive feasibility
certificate for any candidate triple.

The radicality proof reveals a layered structure:

\begin{proposition}[Layered Filtration]\label{prop:filtration}
The assembly ideal induces a filtration of quotient rings:
\begin{align*}
  R &\;\twoheadrightarrow\; R/J_{\mathrm{bool}} \;\cong\; \kk^{2^n}
    &&\text{(Boolean hypercube)}, \\
  &\;\twoheadrightarrow\; R/(J_{\mathrm{bool}} + J_{\mathrm{sel}})
    \;\cong\; \kk^{|F| \cdot |H|^2}
    &&\text{(selection simplex)}, \\
  &\;\twoheadrightarrow\; R/I \;\cong\; \kk^{|\Var(I)|}
    &&\text{(feasible assemblies)}.
\end{align*}
Each quotient is a product of copies of~$\kk$, and each surjection
corresponds to killing coordinate factors.  The compatibility ideal
$K_{\mathrm{compat}}$ kills exactly $|F| \cdot |H|^2 - |\Var(I)|$
factors in the passage from layer~2 to layer~3.
\end{proposition}

\begin{corollary}[Primary Decomposition]\label{cor:primary}
The primary decomposition of $I$ is
\[
  I \;=\; \bigcap_{p \,\in\, \Var(I)} \mathfrak{m}_p
\]
where $\mathfrak{m}_p = \langle x_i - p_i : i = 1, \ldots, n \rangle$
is the maximal ideal of the point~$p$.
\end{corollary}

\begin{proof}
A zero-dimensional radical ideal in a polynomial ring over a
field has primary decomposition given by the intersection of the
maximal ideals of its $\overline{\kk}$-points
(\cite{Eisenbud95}, Corollary~2.12 and Proposition~3.10;
\cite{CLO15}, Chapter~4, \S7).  All points of $\Var(I)$ are
$\kk$-rational by \cref{prop:zero-dim}, so the maximal ideals
are the stated ones.
\end{proof}

\subsection{Generator counts}

The number of generators of each component is:
\begin{align*}
  |J_{\mathrm{bool}}| &= n = r + 2m, \\
  |J_{\mathrm{sel}}|  &= 3, \\
  |K_{\mathrm{compat}}| &= r \cdot m^2 - \sum_{i=1}^r |\Pairs(f_i)|.
\end{align*}

\begin{example}\label{ex:clickchem}
For the standard bioorthogonal click chemistry system with
$r = 8$ families (SPAAC, CuAAC, IEDDA, Thiol-Maleimide, Oxime,
Hydrazone, Staudinger, Thiol-ene), $m = 17$ handle types, and
$\sum |\Pairs(f_i)| = 30$ compatible pairs, we have
$n = 42$ variables and $|I| = 42 + 3 + (8 \cdot 289 - 30) = 2327$
generators.  The variety has $|\Var(I)| = 30$ points.
\end{example}

\begin{example}[Running example: Kadcyla redesign by dual click]\label{ex:kadcyla}
We instantiate the framework on the ADC design problem described
in \cref{sec:intro}, using as our concrete working conjugate the
dual-click redesign of Kadcyla depicted in \cref{fig:kadcyla}:
the maytansinoid payload DM1, a bifunctional PEG linker carrying
a tetrazine at one end and an azide at the other, and an
acetyl-lysine stand-in for the engineered residue on
trastuzumab.  Trastuzumab carries a cyclopropene handle installed
at that engineered residue; the payload DM1 carries a DBCO
handle; and the two clicks join them via the linker, replacing
the single thioether junction of the original Kadcyla conjugate
with two bioorthogonal junctions.
The assembly is a two-step plan
\begin{align*}
  (f_1, h_1, h'_1)
    &\;=\; (\mathrm{IEDDA},\; \text{cyclopropene},\; \text{tetrazine}),
    &&\text{step 1: Ab + linker}, \\
  (f_2, h_2, h'_2)
    &\;=\; (\mathrm{SPAAC},\; \text{azide},\; \text{DBCO}),
    &&\text{step 2: linker + DM1}.
\end{align*}
Each step is a feasible triple in $\Var(I)$ for
the click chemistry instance of \cref{ex:clickchem}: the pair
$(\text{cyclopropene}, \text{tetrazine}) \in \Pairs(\mathrm{IEDDA})$,
and $(\text{azide}, \text{DBCO}) \in \Pairs(\mathrm{SPAAC})$.  The
two steps are orthogonal because the handle set
$\{\text{cyclopropene},\allowbreak \text{tetrazine},\allowbreak
\text{azide},\allowbreak \text{DBCO}\}$
contains no cross-reactive pair: cyclopropene and tetrazine are
diagnostic for IEDDA (\cref{def:diagnostic}; see
\cref{sec:elimination}), and azide and DBCO do not react with
either of them in the absence of copper or strain mismatch.
The plan is therefore a point of the \emph{multi-step variety}
$\Var(I^{(2)}) \cap \{0,1\}^{2n}$ defined in
\cref{sec:multistep}; the four handles span an edge of the
orthogonality graph $G_\perp$ (\cref{thm:ortho}), which
illustrates why this design is feasible and, as we will see, why
a ninth reaction family would have to satisfy a specific
algebraic condition (\cref{prop:omega5}) to extend such
a plan to three orthogonal steps.  We return to this example after
each principal theorem to interpret its algebraic content on this
concrete ADC design problem.
\end{example}

\section{Gr\"obner Bases and the Fibre Decomposition}\label{sec:groebner}

\subsection{The shape lemma}

Since $I$ is a radical zero-dimensional ideal, the reduced
Gr\"obner basis of $I$ in lex order has triangular form
(the \emph{shape lemma}; see~\cite{CLO15}, Chapter~2, \S6).

\begin{theorem}[Shape Lemma for Assembly Ideals]\label{thm:shape}
Let $\prec$ be the lex ordering on $R$ with
$f_1 \succ \cdots \succ f_r \succ h_1 \succ \cdots \succ h_m
\succ h'_1 \succ \cdots \succ h'_m$.
The reduced Gr\"obner basis $\GB_\prec(I)$ has exactly $n$ elements,
each introducing one new leading variable.  The last element is a
univariate Boolean polynomial $h'^2_m - h'_m$.
\end{theorem}

\begin{proof}
The ideal $I$ is zero-dimensional and radical over an infinite
field~$\kk$.  The general shape lemma (cf.~\cite{CLO15}) gives
$|\GB_\prec(I)| = n$ with triangular leading terms.  The Boolean
constraint $h'^2_m - h'_m$ is the unique element whose leading
monomial involves only the last variable.
\end{proof}

Point enumeration by back-substitution from the shape-lemma basis
is $O(|\Var(I)| \cdot n)$ once the basis has been computed.

\subsection{Family-fibre decomposition}

The key structural property of the assembly ideal is that the
compatibility generators for distinct families share no variables
beyond the Boolean and selection constraints.  The selection
constraint $\sum f_i = 1$ partitions the variety into fibres
indexed by~$F$.

\begin{proposition}[Fibre Decomposition]\label{prop:fibre}
The variety decomposes as a disjoint union of family fibres:
\[
  \Var(I) \;=\; \bigsqcup_{f \in F} \Var(I_f)
\]
where $I_f$ is the fibre ideal obtained by setting $f = 1$ and
all other family variables to~$0$.  The fibres are independent:
$|\Var(I)| = \sum_{f \in F} |\Pairs(f)|$.
\end{proposition}

\begin{proof}
The selection constraint $\sum f_i = 1$ with the Boolean
constraints implies that exactly one $f_i$ equals~1 at any
point of~$\Var(I)$.  Substituting $f_i = 1$, $f_j = 0$ for
$j \neq i$ into $K_{\mathrm{compat}}$ gives the fibre ideal
$I_{f_i} \subset \kk[H, H']$, whose variety is
$\{(h, h') : (h, h') \in \Pairs(f_i)\}$.  The disjointness
of the fibres follows from the selection constraint.
\end{proof}

\begin{remark}\label{rem:fibre-scalability}
The fibre decomposition reduces Gr\"obner basis computation
from a single ring with $n = r + 2m$ variables to $r$~independent
computations in $2m$ variables each.  For our system this
replaces one 42-variable computation with eight 34-variable
computations.
\end{remark}

\subsection{The extension theorem}

The fibre decomposition yields an exact formula for how $\Var(I)$
grows under adjunction of a new family.

\begin{theorem}[Functorial Extension]\label{thm:extension}
Let $I$ be the assembly ideal for families $F$ with handle universe~$H$,
and let $f'$ be a new family with $k$~compatible handle pairs
$\Pairs(f') \subset H \times H$.  Let $I'$ be the assembly
ideal for $F \cup \{f'\}$ over~$H$.  Then:
\[
  |\Var(I')| \;=\; |\Var(I)| + k.
\]
Moreover, $\Var(I) \subset \Var(I')$ as a sub-configuration
(the existing feasible triples are preserved).
\end{theorem}

\begin{proof}
By the fibre decomposition (\Cref{prop:fibre}),
$\Var(I') = \bigsqcup_{f \in F \cup \{f'\}} \Var(I'_f)$.
For every existing family $f \in F$, the fibre ideal $I'_f$
is identical to~$I_f$ because the generators involving $f'$
vanish when $f' = 0$.  Hence $\Var(I'_f) = \Var(I_f)$ for
$f \in F$.  The new fibre $\Var(I'_{f'})$ consists of the
$k$~pairs in $\Pairs(f')$ by definition.  Summing:
\[
  |\Var(I')| = \sum_{f \in F} |\Var(I_f)| + |\Pairs(f')|
  = |\Var(I)| + k.\]
\end{proof}

\begin{remark}\label{rem:extension-omega}
\Cref{thm:extension} governs the single-step variety
$\Var(I)$.  It does \emph{not} imply that $\omega(G_\perp)$
increases: if $f'$ shares a handle with an existing family,
the new edge in~$G_\perp$ may leave the maximum independent
set size unchanged.  See \cref{sec:omega5} for a complete analysis.
\end{remark}

\section{Elimination Ideals and Handle Diagnosticity}\label{sec:elimination}

The fibre decomposition treats families independently; elimination
theory reveals the \emph{inter-family} structure by projecting
$\Var(I)$ onto coordinate subspaces.

\subsection{Three elimination ideals}

\begin{definition}\label{def:elim}
We define three elimination ideals:
\begin{align}
  E_{H,H'} &= I \cap \kk[H, H']
    &&\text{(eliminate families)}, \label{eq:ehh} \\
  E_F &= I \cap \kk[F]
    &&\text{(eliminate all handles)}, \label{eq:ef} \\
  E_{F,H} &= I \cap \kk[F, H]
    &&\text{(eliminate target handles)}. \label{eq:efh}
\end{align}
\end{definition}

By the closure theorem~\cite{CLO15}, the varieties of these
elimination ideals are the Zariski closures of the coordinate
projections of~$\Var(I)$.

Eliminating the family variables collapses the fibre
structure and reveals which handle pairs are compatible
under \emph{any} family.  Concretely, the variety
$\Var(E_{H,H'}) \subseteq \{0,1\}^{2m}$ is the set of
pairs $(h, h')$ that appear in $\Pairs(f)$ for at
least one~$f$.  We read this as the edge set of a bipartite
graph $G_{\mathrm{rxn}}$ on $H \times H'$: the edge
$(h, h')$ exists if and only if the pair is compatible
under some family.  This graph is an invariant of the
compatibility structure, independent of which family
realises a given pair.

The variety $\Var(E_F)$ records the ``active'' families
(those with $\Pairs(f) \neq \varnothing$); in the Boolean
quotient it is simply $\{e_f : f \in F,\; \Pairs(f) \neq
\varnothing\}$.

The elimination ideal $E_{F,H}$ is the most structured: it
detects \emph{handle diagnosticity}.

\subsection{The diagnosticity theorem}

\begin{definition}\label{def:diagnostic}
For a handle $h \in H$, define
$\Fam(h) = \{f \in F : \exists\, h' \text{ s.t.\ }
(h, h') \in \Pairs(f)\}$.
A handle $h$ is \emph{diagnostic} if $|\Fam(h)| = 1$.
\end{definition}

A diagnostic handle has a unique pre-image under the family
map: $|\Fam(h)| = 1$ means the source handle determines the
family coordinate of every point in~$\Var(I)$ supported on~$h$.

\begin{theorem}[Handle Diagnosticity]\label{thm:diagnostic}
Let $h \in H$ be a diagnostic handle with $\Fam(h) = \{f_0\}$.
Then the polynomial $h \cdot (1 - f_0)$ lies in the elimination
ideal $E_{F,H} = I \cap \kk[F, H]$.

Conversely, if $h$ is non-diagnostic ($|\Fam(h)| > 1$), then
no polynomial of the form $h \cdot (1 - f_0)$ for any $f_0 \in F$
lies in~$E_{F,H}$.
\end{theorem}

\begin{proof}
If $\Fam(h) = \{f_0\}$, then at every point
$p = (f, h^*, h'^*) \in \Var(I)$ with $h^* = h$, we must have
$f = f_0$.  On $\Var(I) \subseteq \{0,1\}^n$, this means
$h \cdot (1 - f_0)$ vanishes at every point of~$\Var(I)$.
Since $I$ is radical and $\Var(I)$ consists of $\kk$-rational
points (\cref{prop:zero-dim}), by the Strong Nullstellensatz
(\cite{CLO15}, Chapter~4, \S2, Theorem~6;
\cite{Eisenbud95}, Theorem~1.6) vanishing on $\Var(I)$ implies membership
in~$I$, and since $h \cdot (1 - f_0) \in \kk[F, H]$, it
lies in $I \cap \kk[F, H] = E_{F,H}$.

For the converse, suppose $|\Fam(h)| > 1$ and fix any
$f_0 \in F$.  If $f_0 \notin \Fam(h)$, then every point
of~$\Var(I)$ with $h = 1$ has $f_0 = 0$, so
$h \cdot (1 - f_0) = h$ on $\Var(I)$; but $h$ does not vanish
on~$\Var(I)$ (any triple with source handle~$h$ witnesses this),
so $h \cdot (1 - f_0) \notin I$.  If $f_0 \in \Fam(h)$, then
since $|\Fam(h)| > 1$ there exists $f_a \in \Fam(h)$ with
$f_a \neq f_0$ and a triple $(f_a, h, h'_a) \in \Var(I)$;
at this point $h = 1$ and $f_0 = 0$, giving
$h \cdot (1 - f_0) = 1 \neq 0$, so again
$h \cdot (1 - f_0) \notin I \supseteq E_{F,H}$.
\end{proof}

\paragraph{Interpretation.}
\Cref{thm:diagnostic} characterises handle diagnosticity via the
elimination ideal: $h$ is diagnostic if and only if there exists
$f_0 \in F$ with $h \cdot (1 - f_0) \in E_{F,H}$.
The non-diagnostic handles are precisely those with multiplicity $> 1$
in the family map, and they are the source of all inter-family coupling
in the multi-step ideal (\cref{sec:multistep}).  In the bioorthogonal
instance, the five non-diagnostic handles---azide ($|\Fam| = 3$),
aldehyde, ketone, norbornene, and thiol (each $|\Fam| = 2$)---generate
the four bottleneck corridors that bound $\omega(G_\perp)$.

\begin{example}\label{ex:diagnostic}
In the standard 8-family system, 12 of 17 handles are
diagnostic.  The five non-diagnostic handles are:
\begin{center}
\begin{tabular}{ll}
\toprule
Handle & Families \\
\midrule
azide & SPAAC, CuAAC, Staudinger \\
aldehyde & Oxime, Hydrazone \\
ketone & Oxime, Hydrazone \\
norbornene & IEDDA, Thiol-ene \\
thiol & Thiol-Maleimide, Thiol-ene \\
\bottomrule
\end{tabular}
\end{center}
The elimination ideal $E_{F,H}$ for the Oxime--Hydrazone
subsystem contains the polynomial
\[
  f_{\mathrm{Hyd}} \cdot h_{\mathrm{ald}}
  + f_{\mathrm{Hyd}} \cdot h_{\mathrm{ket}}
  - f_{\mathrm{Hyd}} + h_{\mathrm{hydrazine}} \;=\; 0
\]
which encodes the constraint: \emph{the family variable
$f_{\mathrm{Hyd}}$ equals~1 only if the source handle is a
carbonyl (aldehyde or ketone) and the target handle is hydrazine}.
\end{example}

\paragraph{On the running example.}
Three of the four handles of the Kadcyla plan---cyclopropene,
tetrazine, and DBCO---are diagnostic: each appears in only one
family (IEDDA, IEDDA, and SPAAC respectively).  \Cref{thm:diagnostic}
then yields three algebraic certificates,
\[
  \text{cyp}\,(1 - f_{\mathrm{IEDDA}}),\quad
  \text{tetz}\,(1 - f_{\mathrm{IEDDA}}),\quad
  \text{DBCO}\,(1 - f_{\mathrm{SPAAC}}) \;\in\; E_{F,H}.
\]
The remaining handle, azide, is \emph{non}-diagnostic:
it appears in three families (SPAAC, CuAAC, Staudinger),
so its presence in a design does not by itself pin down
the family.  This is why dual-click ADC protocols pair
azide with DBCO rather than a terminal alkyne: DBCO
is diagnostic for SPAAC, so the family of the azide
step is fixed by its partner.  The algebra thereby recovers, as a
theorem about $E_{F,H}$, the practical chemistry argument for
the choice of linker.

\section{Algebraic Statistics on the Assembly Variety}\label{sec:statistics}

\subsection{The log-linear model}

We equip $\Var(I)$ with a probability distribution by assigning
a non-negative weight to each point.  The natural parametric
family is the \emph{log-linear model}:
\[
  P(f, h, h') \;=\; \frac{1}{Z}\,
    \exp(\theta_f + \theta_h + \theta_{h'})
  \qquad\text{for } (f, h, h') \in \Var(I),
\]
where $\theta = (\theta_{f_1}, \ldots, \theta_{f_r},
\theta_{h_1}, \ldots, \theta_{h_m},
\theta_{h'_1}, \ldots, \theta_{h'_m}) \in \RR^n$
are the natural parameters and
$Z = \sum_{(f,h,h') \in \Var(I)} \exp(\theta_f + \theta_h + \theta_{h'})$
is the normalising constant (partition function).
The parametrisation has a direct probabilistic reading:
setting $p_f = e^{\theta_f}$, $p_h = e^{\theta_h}$, and
$p_{h'} = e^{\theta_{h'}}$, the unnormalised weight of a
feasible triple is the product $p_f \cdot p_h \cdot p_{h'}$
of three marginal propensities---one per coordinate---and
$Z = \sum_{(f,h,h') \in \Var(I)} p_f \, p_h \, p_{h'}$
ensures that the weights sum to~1.
In other words, the natural parameters $\theta$ are the
\emph{log-propensities} and $Z$ is the sum over feasible
triples of their coordinate-wise products.

The log-linear model is the canonical choice for three reasons.
First, on a finite support set such as $\Var(I) \subseteq \{0,1\}^n$,
it is the maximum-entropy distribution subject to matching the three
coordinate marginals (family, source handle, target handle)
(cf.~\cite{Csiszar75,WainwrightJordan08}); it therefore assumes
the least structure beyond what the marginals impose.  Second, its additive sufficient statistic
$\theta_f + \theta_h + \theta_{h'}$ mirrors the fibre decomposition
of~$\Var(I)$: within each family fibre, $h$ and~$h'$ contribute
independently, so the model respects the algebraic product structure
of~$R/I$ established in \cref{sec:ring}.  Third, log-linear models
on integer design matrices are precisely the objects whose
implicit descriptions are toric ideals~\cite{PS05,DSS09}; choosing a
different parametric family would sever the connection to toric
geometry and the binomial-relation analysis that follows.

\begin{definition}[Design Matrix]\label{def:design}
The \emph{design matrix} $A \in \{0,1\}^{|\Var(I)| \times n}$
has rows indexed by feasible triples.  The row for $(f_i, h_a, h'_b)$
has a~1 in positions $i$, $r + a$, and $r + m + b$, and 0~elsewhere.
\end{definition}

\begin{proposition}\label{prop:rank}
For the 8-family, 17-handle click-chemistry instance,
$A \in \{0,1\}^{30 \times 42}$ with $\rank(A) = 28$.
The model dimension is $\dim(M) = \rank(A) - 1 = 27$.
\end{proposition}

\begin{proof}
Direct computation of the rank of the $30 \times 42$ binary
matrix over~$\QQ$.
\end{proof}

The model dimension equals the dimension of the linear span
$\langle \Var(I) \rangle$ of the 30~points of~$\Var(I)$
inside~$\RR^{42}$.

\subsection{The toric ideal}

The log-linear model $M \subset \Delta_{29}$ is the image
of the monomial map
\[
  \varphi \colon (\RR_{>0})^n \to \Delta_{29}, \qquad
  (w_f, w_h, w_{h'}) \mapsto
  \Bigl(\frac{w_{f_i} w_{h_a} w_{h'_b}}{Z}\Bigr)_{(i,a,b) \in \Var(I)}.
\]

\begin{definition}
The \emph{toric ideal} of the model is
$I_A = \ker(\varphi) \subset \kk[p_1, \ldots, p_{30}]$,
the ideal of polynomial relations among the
``probability coordinates'' $p_\alpha$, $\alpha \in \Var(I)$.
\end{definition}

The toric ideal is generated by binomials
$p_\alpha p_\beta - p_\gamma p_\delta$ whenever the
row sums $a_\alpha + a_\beta = a_\gamma + a_\delta$
in~$\ZZ^n$.

\begin{theorem}[Toric Ideal Structure for the 8-family instance]\label{thm:toric}
The toric ideal $I_A$ of the assembly log-linear model for the
8-family, 17-handle system is generated by exactly two binomials:
\begin{align*}
  p_{(\mathrm{Ox},\,\mathrm{ald},\,\mathrm{aox})} \cdot
  p_{(\mathrm{Hy},\,\mathrm{ket},\,\mathrm{hyz})}
  &\;=\;
  p_{(\mathrm{Ox},\,\mathrm{ket},\,\mathrm{aox})} \cdot
  p_{(\mathrm{Hy},\,\mathrm{ald},\,\mathrm{hyz})}, \\[3pt]
  p_{(\mathrm{Ox},\,\mathrm{aox},\,\mathrm{ald})} \cdot
  p_{(\mathrm{Hy},\,\mathrm{hyz},\,\mathrm{ket})}
  &\;=\;
  p_{(\mathrm{Ox},\,\mathrm{aox},\,\mathrm{ket})} \cdot
  p_{(\mathrm{Hy},\,\mathrm{hyz},\,\mathrm{ald})}.
\end{align*}
Both binomials correspond to the exchange of two shared
handle types between the same pair of families.
\end{theorem}

\begin{proof}
The toric ideal $I_A$ lives in $\kk[p_1, \ldots, p_{30}]$ and
has codimension $N - \rank(A) = 30 - 28 = 2$.  It therefore
suffices to exhibit two independent elements that generate
the full toric lattice $L = \ker_{\ZZ}(A^T) \subset \ZZ^{30}$,
since $I_A$ is generated by the binomials corresponding to a
lattice basis of~$L$ when this basis forms a complete intersection
(see~\cite{Sturmfels96}, Lemma~4.1).

Enumerating all $\binom{30}{2}$ pairs of rows, exactly two
satisfy the row-sum condition
$a_\alpha + a_\beta = a_\gamma + a_\delta$;
these are the two stated binomials, with exponent vectors
$e_1, e_2 \in \ZZ^{30}$.  Since $\rank(L) = 2$ and
$\rank\{e_1, e_2\} = 2$, it remains to verify that $e_1, e_2$
generate~$L$ as a $\ZZ$-module (not merely a sublattice).  This
holds if and only if $\gcd$ of all $2 \times 2$ minors of the
matrix $(e_1 \mid e_2)^T$ equals~1, which we verified
computationally.  Hence $I_A = \langle b_1, b_2 \rangle$ is a
codimension-2 complete intersection.
\end{proof}

\paragraph{Interpretation.}
The sparseness of the toric ideal (2 generators for a 30-point
model) is a structural measure of the low redundancy in the
compatibility relation: each feasible triple is nearly uniquely
determined by its marginals.  Both binomials arise from the same
algebraic phenomenon---the exchange of two handle types between
two families sharing a common handle set.  In the chemical
instantiation, this corresponds to \emph{carbonyl equivalence}
in imine-type ligations (interchangeability of aldehyde and
ketone between the Oxime and Hydrazone families), a symmetry
known empirically but here derived from the ideal structure.
No other family pair produces a binomial relation.

\paragraph{On the running example.}
Neither binomial of the toric ideal involves the IEDDA or SPAAC
coordinates that the Kadcyla plan occupies, so the two steps of
the running example sit in a \emph{toric-rigid} region of the
model: their probabilities are determined, up to the global
normalisation, by their individual family and handle marginals,
with no hidden binomial exchange.  Combined with $\mathrm{ML\,degree} = 1$
(\cref{thm:mldeg}), this means that if one observed a dataset of
Kadcyla-type ADC designs in the wild, the MLE of the
family/handle propensities would be a unique closed-form rational
function of the empirical counts---no iterative algorithm or
multimodality to worry about.

\subsection{Mixture of independence models}

The fibre decomposition induces a mixture structure on
the statistical model:
\[
  P(f, h, h') \;=\; \pi_f \cdot P_f(h, h')
\]
where $\pi_f = w_f / \sum w_f$ and
$P_f(h, h') = w_h \cdot w_{h'} / Z_f$ is the conditional
distribution within fibre~$f$.  Within each fibre,
$h$ and $h'$ are conditionally independent given~$f$.
The marginal model (integrating over~$f$) is a
\emph{mixture of independence models} --- a well-studied
object in algebraic statistics~\cite{PS05}.

\begin{proposition}\label{prop:not-product}
For the 8-family, 17-handle instance, the variety $\Var(I)$ is not
a Cartesian product:
\[
  \Var(I) \;\subsetneq\; \pi_F(\Var(I)) \times \pi_H(\Var(I)) \times \pi_{H'}(\Var(I)),
\]
with a deficit of $|F| \cdot |H|^2 - |\Var(I)| = 2282$ triples.
\end{proposition}

\begin{proof}
Direct enumeration of $\Var(I)$ and the three coordinate
projections, verified computationally.
\end{proof}

This non-product structure is the source of all
non-trivial algebraic statistics on~$\Var(I)$: it is
what makes the toric ideal non-trivial (two generators
rather than zero).

\subsection{ML degree of the log-linear model}\label{sec:mldeg}

The \emph{ML degree}---the number of complex
critical points of the likelihood function for generic data---is
a fundamental algebraic invariant of the model
(cf.~\cite{DSS09}, Chapter~7).

\begin{theorem}[ML Degree for the 8-family instance]\label{thm:mldeg}
$\mathrm{MLdeg}(M) = 1$.
\end{theorem}

\begin{proof}
For the 8-family, 17-handle system ($n = 42$), the model $M$ is
a regular exponential family with natural parameter
$\theta \in \RR^{42}$ and sufficient statistic $T(x) = A^T x$,
where $A$ is the $30 \times 42$ design matrix.
By a standard result in exponential family theory
(see~\cite{BN78}, Theorem~9.1, or~\cite{DSS09}, \S7.1), the
log-partition function
$\psi(\theta) = \log \sum_{\alpha \in \Var(I)} \exp(a_\alpha^T \theta)$
is strictly convex on the identifiable parameter space
(the affine span of the columns of~$A^T$,
which has dimension $\rank(A) = 28$).
The negative log-likelihood $\ell(\theta) = \psi(\theta) - \theta^T A^T (u/n)$
is therefore strictly convex in this
28-dimensional space.  A strictly convex function on
a convex domain has at most one critical point.  For generic
data $u$ in the interior of the marginal polytope~$\conv(A)$,
the MLE exists (by compactness of the probability simplex and
openness of the exponential family) and is the unique critical
point.  Hence $\mathrm{MLdeg}(M) = 1$.

We verified this numerically by solving the MLE via Newton's
method from 50 random initialisations across 10 generic datasets
sampled from $\mathrm{Dir}(\mathbf{1}_{30})$: all 500~runs converged
to the same unique critical point (moment-matching error $< 10^{-10}$).
\end{proof}

\paragraph{Interpretation.}
$\mathrm{MLdeg}(M) = 1$ means the MLE is a rational function of
the data, computable in closed form.  The structural reason is the
fibre decomposition: conditioned on the family indicator, each
per-family model is a complete-independence model with
$\mathrm{MLdeg} = 1$, and the exponential-family structure
ensures the global MLE inherits this property.  A connection to the product formulae for staged mixture
models studied in~\cite{DSS09} (Chapter~7) warrants further
investigation.

\section{Multi-step Assembly and Orthogonality}\label{sec:multistep}

\subsection{The multi-step ideal}

A \emph{$k$-step assembly plan} is a $k$-tuple of feasible
triples satisfying pairwise family-distinctness and a
cross-reactivity exclusion condition.  We formalise
this as a single ideal in a $k$-fold product ring.

\begin{definition}[Multi-step Assembly Ring]\label{def:multistep}
For $k \geq 2$, the \emph{$k$-step assembly ring} is
\[
  R^{(k)} = \kk\bigl[
    F^{(1)}, H^{(1)}, H'^{(1)}, \;\ldots,\;
    F^{(k)}, H^{(k)}, H'^{(k)}
  \bigr]
\]
with $kn$ variables, equipped with the \emph{multi-step ideal}
\[
  I^{(k)} = \sum_{j=1}^k I^{(j)}
    + J_{\mathrm{orth}} + J_{\mathrm{cross}}
\]
where:
\begin{enumerate}[nosep,label=(\alph*)]
\item $I^{(j)}$ is a copy of the single-step ideal in the
  $j$-th block of variables.
\item $J_{\mathrm{orth}} = \langle f_i^{(s)} \cdot f_i^{(t)} :
  s \neq t,\; \forall\, i \rangle$ (family orthogonality).
\item $J_{\mathrm{cross}} = \langle h_a^{(s)} \cdot h'^{(t)}_b :
  s \neq t,\; (h_a, h_b) \in \bigcup_f \Pairs(f) \rangle$
  (cross-reactivity).
\end{enumerate}
\end{definition}

For $k \geq 2$, the cross-reactivity constraints
$J_{\mathrm{cross}}$ couple the fibre ideals of different
steps, destroying the product structure of the single-step
variety.

\begin{proposition}[Fibre Decomposition Failure]\label{prop:fibre-break}
For $k \geq 2$, the variety $\Var(I^{(k)})$ does not
decompose as a Cartesian product of single-step varieties:
\[
  \Var(I^{(k)}) \;\subsetneq\;
  \Var(I)^k \;=\;
  \underbrace{\Var(I) \times \cdots \times \Var(I)}_{k}.
\]
\end{proposition}

\begin{proof}
It suffices to exhibit a pair $(t_1, t_2) \in \Var(I)^2$ that
violates $J_{\mathrm{cross}}$ and hence does not lie in
$\Var(I^{(2)})$.  Take
$t_1 = (\mathrm{SPAAC},\, \mathrm{azide},\, \mathrm{strained\_alkyne})$
and
$t_2 = (\mathrm{CuAAC},\, \mathrm{azide},\, \mathrm{alkyne})$.
Both are feasible triples (different families), so
$(t_1, t_2) \in \Var(I)^2$.  However, the source handle of step~1
is azide and the target handle of step~2 is alkyne; since
$(\mathrm{azide}, \mathrm{alkyne}) \in \Pairs(\mathrm{CuAAC})$,
the cross-reactivity generator
$h_{\mathrm{azide}}^{(1)} \cdot h'^{(2)}_{\mathrm{alkyne}}$
does not vanish at $(t_1, t_2)$, so
$(t_1, t_2) \notin \Var(I^{(2)})$.
\end{proof}

\subsection{The orthogonality graph}

\begin{definition}\label{def:orth-graph}
The \emph{orthogonality graph} $G_\perp$ has vertex set
$\Var(I)$ and an edge between triples $t_1$ and $t_2$ if
and only if $t_1$ and $t_2$ are mutually compatible in a
2-step plan: different families, disjoint handles, and no
cross-reactivity.
\end{definition}

The maximum multiplicity of a $k$-step plan equals the
\emph{clique number} $\omega(G_\perp)$.

\begin{theorem}[Maximum Orthogonal Multiplicity]\label{thm:ortho}
For the standard 8-family, 17-handle bioorthogonal system:
\[
  \omega(G_\perp) = 4.
\]
A maximum clique is (one of six):
\begin{center}
\begin{tabular}{lll}
\toprule
Step & Family & Handles $(h \to h')$ \\
\midrule
1 & SPAAC & DBCO $\to$ azide \\
2 & IEDDA & TCO $\to$ tetrazine \\
3 & Thiol-Maleimide & maleimide $\to$ thiol \\
4 & Oxime & aminooxy $\to$ aldehyde \\
\bottomrule
\end{tabular}
\end{center}
In particular, at most 4 of 8 families can be used simultaneously:
$\omega(G_\perp) < |F|$.  Six distinct maximum independent family
sets exist.  The family-level handle-sharing graph
(vertices~$= F$, edge iff shared handle) has 6~edges from four
bottleneck corridors and chromatic number~3
(\Cref{fig:orthogonality}).
\end{theorem}

\begin{proof}
The clique number is computed by exhaustive enumeration of all
$\binom{30}{k}$ subsets of $\Var(I)$ for $k = 2, \ldots, 8$,
checking pairwise compatibility.  Let $N_k$ denote the number
of \emph{unordered} $k$-step orthogonal protocols (i.e.\ sets of
$k$ pairwise-compatible triples from distinct families); each
such set corresponds to $k!$ ordered tuples in~$\Var(I^{(k)})$,
so $|\Var(I^{(k)})| = k!\, N_k$.  We find:
\[
  N_k \;=\; 30,\; 304,\; 1152,\; 1024,\; 0,\; 0,\; 0,\; 0
  \qquad \text{for } k = 1, \ldots, 8.
\]
The maximum clique size is $k = 4$.  The vanishing of
$N_5$ means no 5-step orthogonal plan exists.
\end{proof}

\paragraph{Interpretation.}
The gap $\omega(G_\perp) = 4 < |F| = 8$ is a graph-theoretic
obstruction caused entirely by handle sharing: the five
non-diagnostic handles ($|\Fam(h)| > 1$) induce four corridors
in~$G_\times$ that force mutual exclusion among families.  The six
maximum independent sets (\Cref{tab:invariants}) enumerate all
optimal selections; each contains one family from the azide triangle,
one from the carbonyl pair, and the two conflict-free families.
In the chemical setting, this means at most four distinct reaction
chemistries can be deployed in a multiplexed protocol.

\paragraph{On the running example.}
The Kadcyla plan of \cref{ex:kadcyla} uses two of these
four families---IEDDA and SPAAC---which together form a
2-clique in $G_\perp$ and sit inside each of the six maximum
independent sets (\cref{tab:invariants}).  This is what makes the
dual-click design robust: both IEDDA and SPAAC are
conflict-free, so the protocol inherits all the cross-reactivity
guarantees that the MIS analysis provides, and the remaining
two slots in the maximum clique can be populated with a
\emph{carbonyl} step (Oxime or Hydrazone) and an \emph{azide-route}
step (distinct from the SPAAC step already in use) to assemble
a hypothetical 4-step orthogonal extension of the Kadcyla
protocol.  The obstruction to a \emph{fifth} such step is the
content of the next subsection and of \cref{sec:omega5}.

\subsection{The four bottleneck corridors}

The obstruction to higher clique numbers can be precisely
localised in the cross-reactivity graph.

\begin{definition}\label{def:cross-graph}
The \emph{cross-reactivity graph} $G_\times$ has vertex set
$\Var(I)$ and an edge between $t_1 = (f_1, h_1, h'_1)$ and
$t_2 = (f_2, h_2, h'_2)$ (with $f_1 \neq f_2$) whenever
$(h_1, h'_2) \in \bigcup_f \Pairs(f)$ or
$(h_2, h'_1) \in \bigcup_f \Pairs(f)$.
\end{definition}

\begin{proposition}[Bottleneck Corridors]\label{prop:bottleneck}
For the 8-family instance,
the cross-reactivity graph $G_\times$ has edges distributed in
four bottleneck corridors:
\begin{enumerate}[nosep,label=(\roman*)]
\item The \emph{azide corridor} (three-way): SPAAC $\leftrightarrow$ CuAAC
  (6~pairs), SPAAC $\leftrightarrow$ Staudinger (6~pairs), and
  CuAAC $\leftrightarrow$ Staudinger (4~pairs), all caused by the
  shared handle \emph{azide} ($|\Fam(\text{azide})| = 3$).
\item The \emph{carbonyl corridor}: 8 cross-reactive pairs between
  Oxime and Hydrazone triples, caused by the shared handles
  \emph{aldehyde} and \emph{ketone}
  ($|\Fam(\text{ald})| = |\Fam(\text{ket})| = 2$).
\item The \emph{thiol corridor}: Thiol-Maleimide $\leftrightarrow$
  Thiol-ene, caused by the shared handle \emph{thiol}.
\item The \emph{norbornene corridor}: IEDDA $\leftrightarrow$
  Thiol-ene, caused by the shared handle \emph{norbornene}.
\end{enumerate}
The azide corridor is the densest ($|\Fam(\text{azide})| = 3$);
the carbonyl corridor involves two shared handles
($|\Fam(\text{ald})| = |\Fam(\text{ket})| = 2$).
\end{proposition}

\begin{proof}
Direct enumeration over all $\binom{30}{2}$ cross-family pairs.
\end{proof}

\begin{corollary}\label{cor:cuaac-excluded}
In the 8-family instance, every maximum orthogonal plan must
choose one family from each bottleneck corridor: either SPAAC
or CuAAC (but not both),
and either Oxime or Hydrazone (but not both).  IEDDA and
Thiol-Maleimide participate in all maximum independent sets
because they share no handles with any family that itself
appears in a maximum independent set (their shared handles
with Thiol-ene are immaterial, since Thiol-ene is excluded
from every MIS by its conflicts in the thiol and norbornene
corridors).
\end{corollary}

\subsection{The orthogonality sequence}

For the 8-family instance,
the sequence $(N_k)_{k=1}^{8}$ of unordered protocol counts is:
\[
  30, \;\; 304, \;\; 1152, \;\; 1024, \;\; 0, \;\; 0, \;\; 0, \;\; 0.
\]
This is \emph{non-monotone}: it increases from $k = 1$
to $k = 3$, then decreases sharply.  The peak at $k = 3$
reflects the abundance of 3-step orthogonal designs
($N_3 = 1152$ unordered protocols, distributed over 23 family
subsets), while the steep drop to $N_4 = 1024$ (six family
subsets) and the vanishing at $k = 5$ is the ``orthogonality
cliff'' imposed by the four bottleneck corridors.

\paragraph{Interpretation.}
The orthogonality sequence $(N_k)$ is a complete invariant of the
multiplexing capacity at each level~$k$.  The peak at $k = 3$
($N_3 = 1152$) and the steep drop to $k = 4$ ($N_4 = 1024$)
quantify the combinatorial cost of approaching the clique number;
the vanishing at $k = 5$ is a hard graph-theoretic boundary
imposed by $\omega(G_\perp) = 4$.  In the chemical setting, this
means three-step protocols offer the richest design freedom,
while five-step and higher protocols are infeasible.

\subsection{Emerging families and the $\omega = 5$ barrier}\label{sec:omega5}

The bound $\omega(G_\perp) = 4$ is tight for the current
$|F| = 8$ system.  We now characterise precisely when adjoining
a new family raises the clique number.

\paragraph{Necessary conditions for $\omega = 5$.}
A ninth family $f_9$ raises $\omega(G_\perp)$ to~5 if and only
if $f_9$ is not adjacent to all four families in some maximum
independent set.  Equivalently, the handles of~$f_9$ must be
disjoint from the handles of at least one MIS.  Each of the
six current MIS leaves between 4 and 5 handles unused
(\Cref{tab:available}), so a new family using only
those ``available'' handles would extend the MIS to size~5.

\begin{table}[h!]
\centering
\caption{Available handles per maximum independent set.  A new family
  using handles exclusively from this pool extends the MIS to size~5.}
\label{tab:available}
\resizebox{\textwidth}{!}{%
\begin{tabular}{@{}clc@{}}
\toprule
MIS & Available handles & Count \\
\midrule
1 (SPAAC, IEDDA, Thiol-Mal, Oxime) & alkene, alkyne, hydrazine, phosphine & 4 \\
2 (SPAAC, IEDDA, Thiol-Mal, Hydraz) & alkene, alkyne, aminooxy, phosphine & 4 \\
3 (CuAAC, IEDDA, Thiol-Mal, Oxime) & DBCO, alkene, hydrazine, phosphine, str\_alkyne & 5 \\
4 (CuAAC, IEDDA, Thiol-Mal, Hydraz) & DBCO, alkene, aminooxy, phosphine, str\_alkyne & 5 \\
5 (IEDDA, Thiol-Mal, Oxime, Staud) & DBCO, alkene, alkyne, hydrazine, str\_alkyne & 5 \\
6 (IEDDA, Thiol-Mal, Hydraz, Staud) & DBCO, alkene, alkyne, aminooxy, str\_alkyne & 5 \\
\bottomrule
\end{tabular}}
\end{table}

\paragraph{A graph-theoretic characterisation.}

Two families $f_i, f_j$ are \emph{handle-disjoint} if
$H_{f_i} \cap H_{f_j} = \emptyset$.  A set $S$ of families is a
\emph{maximum independent family set} (MIS) if every pair in~$S$
is handle-disjoint and $|S|$ is maximal.  Every MIS of size~$k$
lifts to at least one $k$-clique in~$G_\perp$, provided the
cross-reactivity constraints $J_{\mathrm{cross}}$ do not block all
triple assignments; conversely, every $k$-clique in~$G_\perp$
projects to~$k$ pairwise handle-disjoint families.

\begin{proposition}[Criterion for $\omega = 5$]\label{prop:omega5}
Let $f_9$ be a new family with handle set $H_9 \subseteq H$.
\begin{enumerate}[nosep,label=(\alph*)]
\item\label{item:omega5-nec}
  \textup{(Necessary condition.)}
  If $\omega(G_\perp \cup \{f_9\}) = 5$, then there exists a
  maximum independent set $S$ of $G_\perp$ with $|S| = 4$ and
  $H_9 \cap \bigl(\bigcup_{f \in S} H_f\bigr) = \emptyset$.
\item\label{item:omega5-novel}
  \textup{(Sufficient condition: novel handles.)}
  If $H_9 \cap H = \emptyset$ (i.e.\ $f_9$ introduces handles not
  present in any existing family), then
  $\omega(G_\perp \cup \{f_9\}) = 5$.
\item\label{item:omega5-general}
  \textup{(General sufficient condition.)}
  Suppose there exists a maximum independent set $S$ with
  $H_9 \cap \bigl(\bigcup_{f \in S} H_f\bigr) = \emptyset$, and
  moreover, for every handle $a \in H_9$ and every handle
  $b \in \bigcup_{f \in S} H_f$, neither $(a, b)$ nor $(b, a)$
  is contained in $\Pairs(g)$ for any family~$g$.  Then
  $\omega(G_\perp \cup \{f_9\}) = 5$.
\end{enumerate}
\end{proposition}

\begin{proof}
\cref{item:omega5-nec}\;
Suppose $\omega(G_\perp \cup \{f_9\}) = 5$.
Then there exists an independent set $S' \subset F \cup \{f_9\}$
with $|S'| = 5$.  Since $\omega(G_\perp) = 4$, the vertex
$f_9$ must belong to~$S'$, so $S' = S \cup \{f_9\}$ where
$S \subset F$ with $|S| = 4$.  Because $S'$ is independent,
$f_9$ is not adjacent to any $f \in S$, i.e.\
$H_9 \cap H_f = \emptyset$ for all $f \in S$.  Hence
$H_9 \cap \bigl(\bigcup_{f \in S} H_f\bigr) = \emptyset$.
Moreover, $S$ is independent in~$G_\perp$ and has size~4,
so $S$ is a maximum independent set of~$G_\perp$.

\cref{item:omega5-novel}\;
Since $H_9 \cap H = \emptyset$, we have $H_9 \cap H_f = \emptyset$
for \emph{every} family~$f$, so condition~\cref{item:omega5-nec}
holds for every MIS~$S$.  To show that a valid 5-step plan
exists, note that because $H_9$ introduces entirely new handle
types, no pair $(a, b)$ with $a \in H_9$ and $b \in H_f$ (or
vice versa) can belong to $\Pairs(g)$ for any existing family~$g$:
such a pair would require $a \in H_g$, contradicting
$H_9 \cap H = \emptyset$.  Therefore no cross-reactivity
arises between $f_9$ and any family in~$S$, and any choice
of one feasible triple per family in $S \cup \{f_9\}$ yields a
valid 5-step plan, giving $\omega(G_\perp \cup \{f_9\}) = 5$.

\cref{item:omega5-general}\;
The argument is identical to~\cref{item:omega5-novel}, with the
explicit no-cross-reactivity hypothesis replacing the stronger
$H_9 \cap H = \emptyset$.  Handle-disjointness from~$S$ ensures
no direct handle sharing; the additional condition that neither
$(a, b)$ nor $(b, a)$ lies in $\Pairs(g)$ for $a \in H_9$,
$b \in \bigcup_{f \in S} H_f$ rules out \emph{third-family
mediated} cross-reactivity in both directions of the
cross-reactivity check (\Cref{def:cross-graph}).
For example, the pair $(\text{thiol}, \text{alkene}) \in
\Pairs(\text{Thiol-ene})$ could otherwise block a triple of~$f_9$
using alkene from coexisting with a Thiol-Maleimide triple using
thiol, even though $f_9$ and Thiol-Maleimide share no handles
directly.  With both conditions satisfied, any assignment of one
feasible triple per family in $S \cup \{f_9\}$ is a valid 5-step
plan in~$\Var(I^{(5)})$.
\end{proof}

\begin{remark}\label{rmk:third-family}
The distinction between parts~\cref{item:omega5-novel}
and~\cref{item:omega5-general} is not merely formal.  In the
current 8-family system, every ``available'' handle for MIS~1
(SPAAC, IEDDA, Thiol-Mal, Oxime)---namely alkene, alkyne,
hydrazine, and phosphine---participates in at least one reactive
pair with a handle of a family in that MIS, mediated by a third
family (Thiol-ene, CuAAC, Hydrazone, Staudinger respectively).
A hypothetical ninth family reusing these handles would satisfy
handle-disjointness from~$S$ yet fail the cross-reactivity
check of~\cref{item:omega5-general}.  All six emerging families
surveyed in \Cref{tab:emerging} are verified computationally.
\end{remark}

\paragraph{Interpretation.}
\Cref{prop:omega5} reduces the question of raising
$\omega$ to a handle-disjointness condition together with the
absence of third-family mediated cross-reactivity.  In particular,
any family with an entirely novel handle set (disjoint from all
17 current types) raises $\omega$ unconditionally
(\cref{item:omega5-novel}).  Conversely, a new family sharing a
handle with a family present in every MIS (in the bioorthogonal
instance: IEDDA or Thiol-Maleimide) cannot raise~$\omega$.

\paragraph{On the running example.}
For the Kadcyla dual-click plan, both IEDDA and SPAAC sit in every
MIS, so \cref{prop:omega5} has an immediate design reading: a
ninth family $f_9$ can extend the protocol to a fifth orthogonal
step only if its handle set is disjoint from
$\{\text{cyclopropene}, \text{tetrazine}, \text{azide},
\text{DBCO}\}$ \emph{and} from the handles of whichever carbonyl
and Thiol-Maleimide choices complete a 4-MIS.  Any family that
shares norbornene (IEDDA's other handle) or thiol (Thiol-Mal's
other handle) is excluded immediately.  A family with entirely
novel handles---SuFEx is the cleanest existing candidate---clears
the criterion.  The algebraic statement, in other words, tells
the chemist exactly where to look when designing the next
generation of Kadcyla-like multi-functional ADCs.

\paragraph{Analysis of emerging reactions.}
We evaluate six candidate emerging reactions against the
MIS availability constraint:

\begin{enumerate}[nosep]
\item \textbf{SuFEx} (sulfur(VI) fluoride exchange):
  uses sulfonyl fluoride and silyl ether, both entirely novel.
  Adds an isolated vertex to~$G_\perp$; joins \emph{all six} MIS.
  $\Rightarrow$ \emph{$\omega$ rises to~5 unconditionally.}

\item \textbf{Photoclick} (tetrazole photolysis $\to$ nitrile
  imine + alkene): shares alkene with Thiol-ene.  Adjacent only
  to Thiol-ene; can join all four MIS containing SPAAC or CuAAC
  (which exclude Thiol-ene by other corridors).
  $\Rightarrow$ \emph{$\omega$ rises to~5.}

\item \textbf{Nitrone--alkene}: shares alkene with Thiol-ene.
  Same adjacency as photoclick.
  $\Rightarrow$ \emph{$\omega$ rises to~5.}

\item \textbf{Sydnone--alkyne}: shares alkyne with CuAAC.
  Adjacent only to CuAAC; can join the four MIS containing
  SPAAC or Staudinger.
  $\Rightarrow$ \emph{$\omega$ rises to~5.}

\item \textbf{Quadricyclane--Ni azide}: entirely novel handles.
  $\Rightarrow$ \emph{$\omega$ rises to~5 unconditionally.}

\item \textbf{Isonitrile--tetrazine} ($[4{+}1]$ cycloadd.):
  shares tetrazine with IEDDA; since IEDDA lies in every
  MIS, the new family is adjacent to all MIS members.
  $\Rightarrow$ \emph{Cannot raise~$\omega$.} The unique
  emerging reaction blocked by graph structure.
\end{enumerate}

\paragraph{Computational verification.}
We verify \Cref{prop:omega5} by explicitly computing
$\omega(G_\perp \cup \{f_9\})$ and a witness MIS for each of the
six emerging reaction candidates (\Cref{tab:emerging}).

\begin{table}[h!]
\centering
\caption{Computational verification of the $\omega = 5$ analysis.
  For each candidate family~$f_9$, we adjoin it to the 8-family system,
  recompute the orthogonality graph, and report $\omega$, the adjacency
  of~$f_9$, and an example MIS of maximum size.}
\label{tab:emerging}
\resizebox{\textwidth}{!}{%
\begin{tabular}{@{}lllcl@{}}
\toprule
Family $f_9$ & Novel handles & Adj.\ to & $\omega'$ & Example MIS \\
\midrule
SuFEx & sulfonyl fl., silyl eth. & (none) & \textbf{5} &
  SPAAC, IEDDA, Thiol-Mal, Oxime, SuFEx \\
Photoclick & diaryl tetrazole & Thiol-ene & \textbf{5} &
  SPAAC, IEDDA, Thiol-Mal, Oxime, Photoclick \\
Nitrone--alkene & nitrone & Thiol-ene & \textbf{5} &
  SPAAC, IEDDA, Thiol-Mal, Oxime, Nitrone-alk. \\
Sydnone--alkyne & sydnone & CuAAC & \textbf{5} &
  SPAAC, IEDDA, Thiol-Mal, Oxime, Sydnone-alk. \\
Quadricyclane & quadricycl., azo dicarb. & (none) & \textbf{5} &
  SPAAC, IEDDA, Thiol-Mal, Oxime, Quadricycl. \\
Isonitrile--Tz & isonitrile & IEDDA & 4 &
  SPAAC, Thiol-Mal, Oxime, Isonitrile-Tz \\
\bottomrule
\end{tabular}}
\end{table}

\noindent
Five of six candidates raise $\omega$ to~5.  The unique exception
is isonitrile--tetrazine, which shares tetrazine with IEDDA; since
IEDDA appears in all six current MIS, the new family cannot extend
any of them.  With SuFEx adjoined, the extended orthogonality
sequence $(N_k)$ becomes $32, 364, 1760, 3328, 2048, 0, \ldots$\,,
exhibiting a new peak at $k = 4$ ($N_4 = 3328$) and
non-vanishing at $k = 5$ ($N_5 = 2048$ five-step protocols),
consistent with the new clique number $\omega = 5$; the sequence
now vanishes at $k = 6$ rather than $k = 5$.

\section{Computational Results}\label{sec:computation}

All computations were performed in SymPy~\cite{SymPy}
over~$\QQ$.  A Python module encodes the combinatorial input
data $(F, H, \Pairs)$ and constructs the polynomial ring and
ideal generators automatically.

\begin{table}[ht]
\centering
\caption{Numerical invariants of the assembly variety.}
\label{tab:invariants}
\begin{tabular}{@{}lr@{}}
\toprule
Invariant & Value \\
\midrule
Families $|F|$ & 8 \\
Handle types $|H|$ & 17 \\
Variables $n = |F| + 2|H|$ & 42 \\
Ideal generators $|I|$ & 2327 \\
Compatibility generators $|K_{\mathrm{compat}}|$ & 2282 \\
Feasible triples $|\Var(I)|$ & 30 \\
$\rank$ of design matrix $A$ & 28 \\
Log-linear model dimension & 27 \\
Diagnostic handles & 12/17 \\
Non-diagnostic (shared) handles & 5 (azide, aldehyde, ketone, norbornene, thiol) \\
Toric ideal generators & 2 \\
Handle-sharing graph edges (family level) & 6 \\
Maximum orthogonal multiplicity $\omega(G_\perp)$ & 4 \\
Maximum independent sets & 6 \\
Chromatic number (family level) & 3 \\
Bottleneck corridors & 4 \\
Selectivity ratio $\sigma = |\Var(I)|/(|F| \cdot |H|^2)$ & 0.0130 \\
\bottomrule
\end{tabular}
\end{table}

\begin{table}[ht]
\centering
\caption{Per-family fibre sizes and selectivity.}
\label{tab:fibres}
\begin{tabular}{@{}lcc@{}}
\toprule
Family $f$ & $|\Pairs(f)|$ & $\sigma_f = |\Pairs(f)| / |H|^2$ \\
\midrule
SPAAC & 4 & 0.0138 \\
CuAAC & 2 & 0.0069 \\
IEDDA & 6 & 0.0208 \\
Thiol-Maleimide & 4 & 0.0138 \\
Oxime & 4 & 0.0138 \\
Hydrazone & 4 & 0.0138 \\
Staudinger & 2 & 0.0069 \\
Thiol-ene & 4 & 0.0138 \\
\midrule
Total & 30 & 0.0130 \\
\bottomrule
\end{tabular}
\end{table}

\subsection{Visualising the variety and its invariants}

We present seven figures that make the algebraic structures
introduced in \cref{sec:ring,sec:multistep} visually
concrete.  All figures are generated programmatically from the
algebraic data using the companion Python module, ensuring
reproducibility and consistency with the numerical results in
Tables~\cref{tab:invariants}--\cref{tab:fibres}.

\paragraph{The interaction graph.}
The Hammersley--Clifford theorem guarantees that any positive
distribution on $\Var(I)$ factorises over the maximal cliques of
the conditional-independence graph.  \Cref{fig:interaction}
displays this graph in its natural tripartite layout
$G = (F \cup H \cup H', E)$.  Each edge connects a family node
to the handles it can use; the 30~maximal cliques correspond
one-to-one with the 30~feasible triples.  The tripartite
structure makes the handle-sharing patterns visually
immediate: azide appears in three family columns
(SPAAC, CuAAC, Staudinger), while cyclopropene, vinyl sulfone,
and the remaining 10~diagnostic handles attach to a single family each.

\begin{figure}[h!]
\centering
\includegraphics[width=0.88\textwidth]{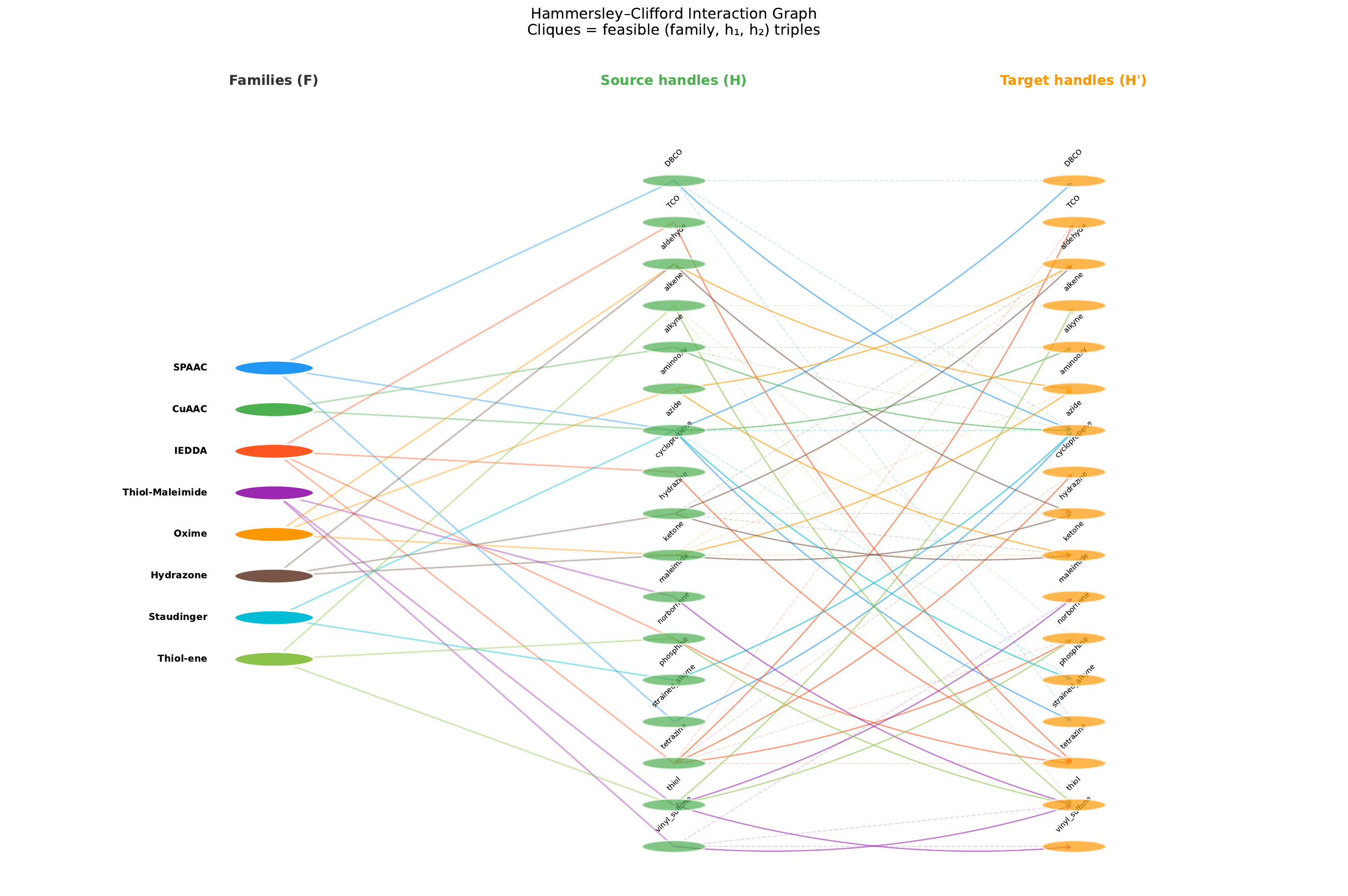}
\caption{The Hammersley--Clifford interaction graph.
  Tripartite layout $G = (F \cup H \cup H', E)$;
  cliques~$=$~feasible triples in~$\Var(I)$.}
\label{fig:interaction}
\end{figure}

\paragraph{Growth and parametric extension.}
\Cref{fig:growth} tracks two complementary aspects of the
variety's size.  The left panel shows the \emph{growth curve}:
$|\Var(I)|$ as families are adjoined one by one, starting from a
single family and building up to the full 8-family system.
Each step adds exactly $|\Pairs(f)|$ new points, confirming
the disjoint-fibre decomposition of \Cref{prop:fibre}.
The right panel tests the \emph{functorial extension theorem}
(\Cref{thm:extension}) by adjoining a hypothetical ninth family with
$p = 1, 2, \ldots, 12$ handle pairs drawn from the existing
handle universe.  In every case, the computed count
$|\Var(I')| = |\Var(I)| + p$ matches the predicted linear growth
exactly.  An important caveat applies: this linear growth measures
single-step feasibility.  It does not extend to the multi-step
orthogonality sequence, because a new family that shares handles
with existing ones introduces edges in~$G_\perp$ without
necessarily raising $\omega(G_\perp)$
(cf.\ \Cref{fig:orthogonality}).

\begin{figure}[h!]
\centering
\includegraphics[width=0.92\textwidth]{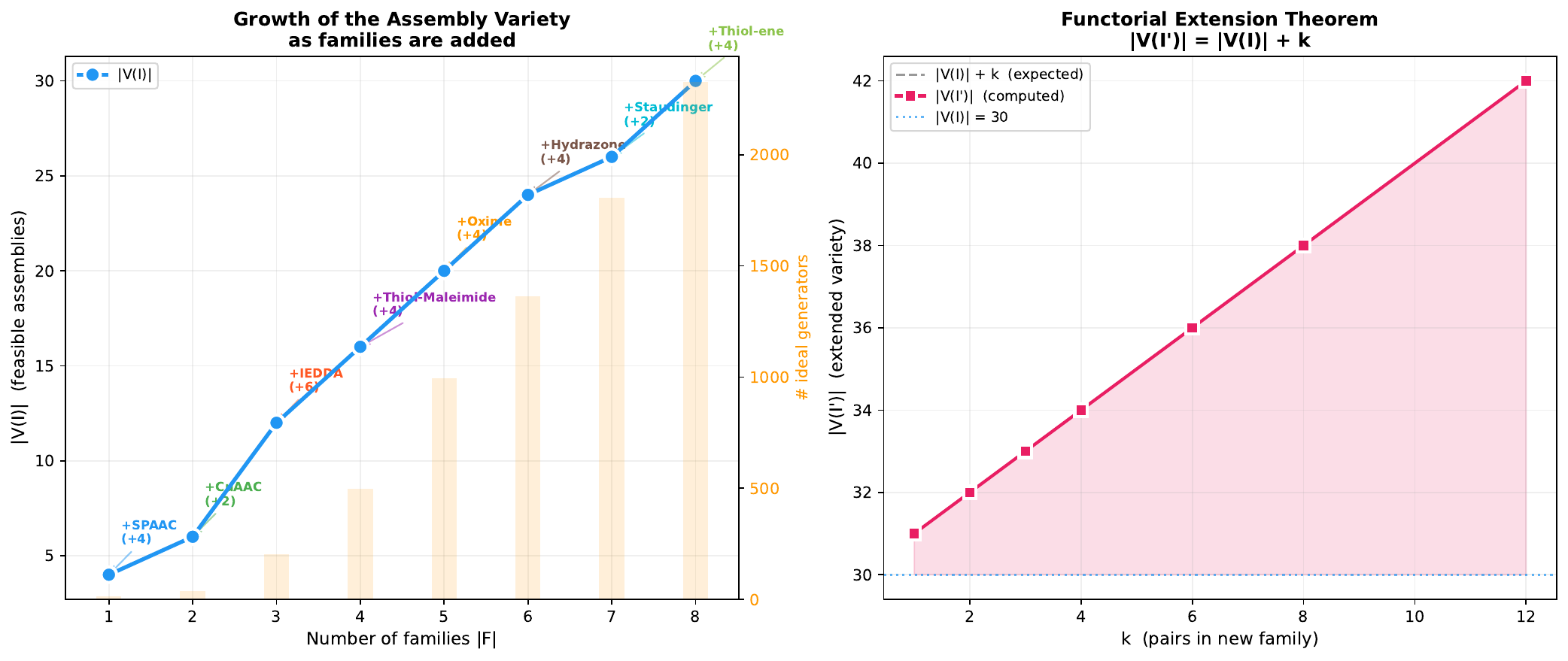}
\caption{Left: growth curve of~$|\Var(I)|$ as families are added.
  Right: parametric extension confirming $|\Var(I')| = |\Var(I)| + p$.
  This linear growth applies to the single-step variety only;
  multi-step orthogonality is governed by $G_\perp$.}
\label{fig:growth}
\end{figure}

\paragraph{Selectivity landscape.}
The selectivity ratio $\sigma = |\Var(I)|/(|F| \cdot |H|^2)$
measures the fraction of $F \times H \times H$ that lies
in~$\Var(I)$.  At $\sigma = 0.013$, fewer than $1.3\%$
of all possible (family, source, target) triples pass the
compatibility constraints.  \Cref{fig:selectivity} decomposes
this global ratio into per-family selectivities~$\sigma_f$ and
displays the sandwich inequality
$|F| \leq |\Var(I)| \leq |F| \cdot |H|^2$.
IEDDA has the highest per-family selectivity ($\sigma_f = 0.021$,
6~pairs) thanks to three dienophile handles
(TCO, norbornene, cyclopropene), while CuAAC and Staudinger have
the lowest ($\sigma_f = 0.007$, 2~pairs each).  The ratio
provides a single scalar measure of how ``constrained'' the
chemistry is, and can serve as an objective function for network
design: a lower~$\sigma$ indicates a more selective, less
promiscuous toolkit.

\begin{figure}[h!]
\centering
\includegraphics[width=0.88\textwidth]{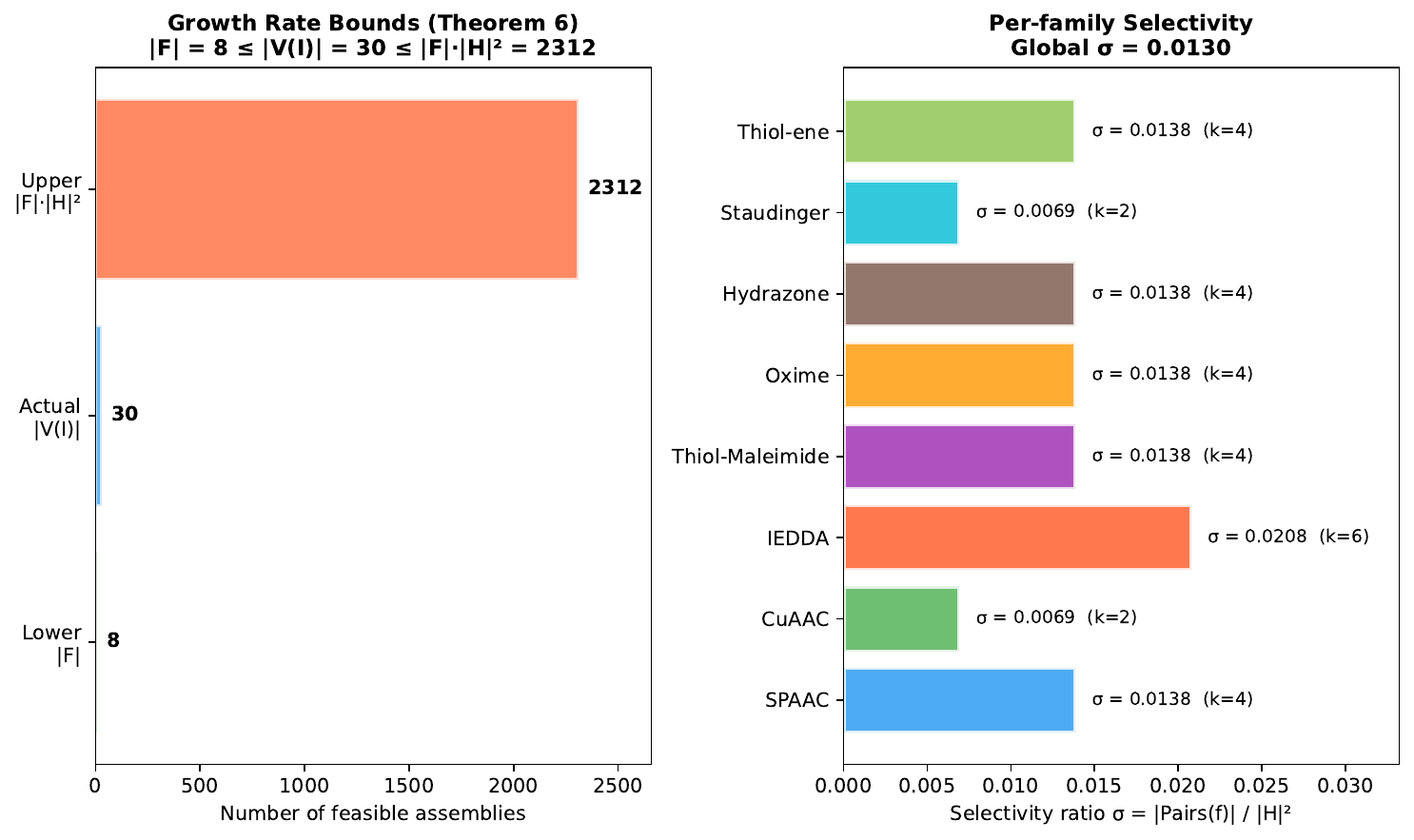}
\caption{Selectivity landscape.  Left: sandwich bounds on~$|\Var(I)|$.
  Right: per-family selectivities~$\sigma_f$.}
\label{fig:selectivity}
\end{figure}

\paragraph{The variety as a point cloud.}
\Cref{fig:points} plots the 30~lattice points of
$\Var(I) \cap \{0,1\}^{42}$, coloured by family membership.
The fibre decomposition $\Var(I) = \bigsqcup_{f} \Var(I_f)$
is manifest as eight non-overlapping clusters: each family's
points form a connected component in the graph on~$\Var(I)$
induced by Hamming adjacency.  The IEDDA cluster (6~points)
is the largest, reflecting its three dienophile handles;
the CuAAC and Staudinger clusters (2~points each) are the
smallest, consistent with their minimal handle pairs.

\begin{figure}[h!]
\centering
\includegraphics[width=0.78\textwidth]{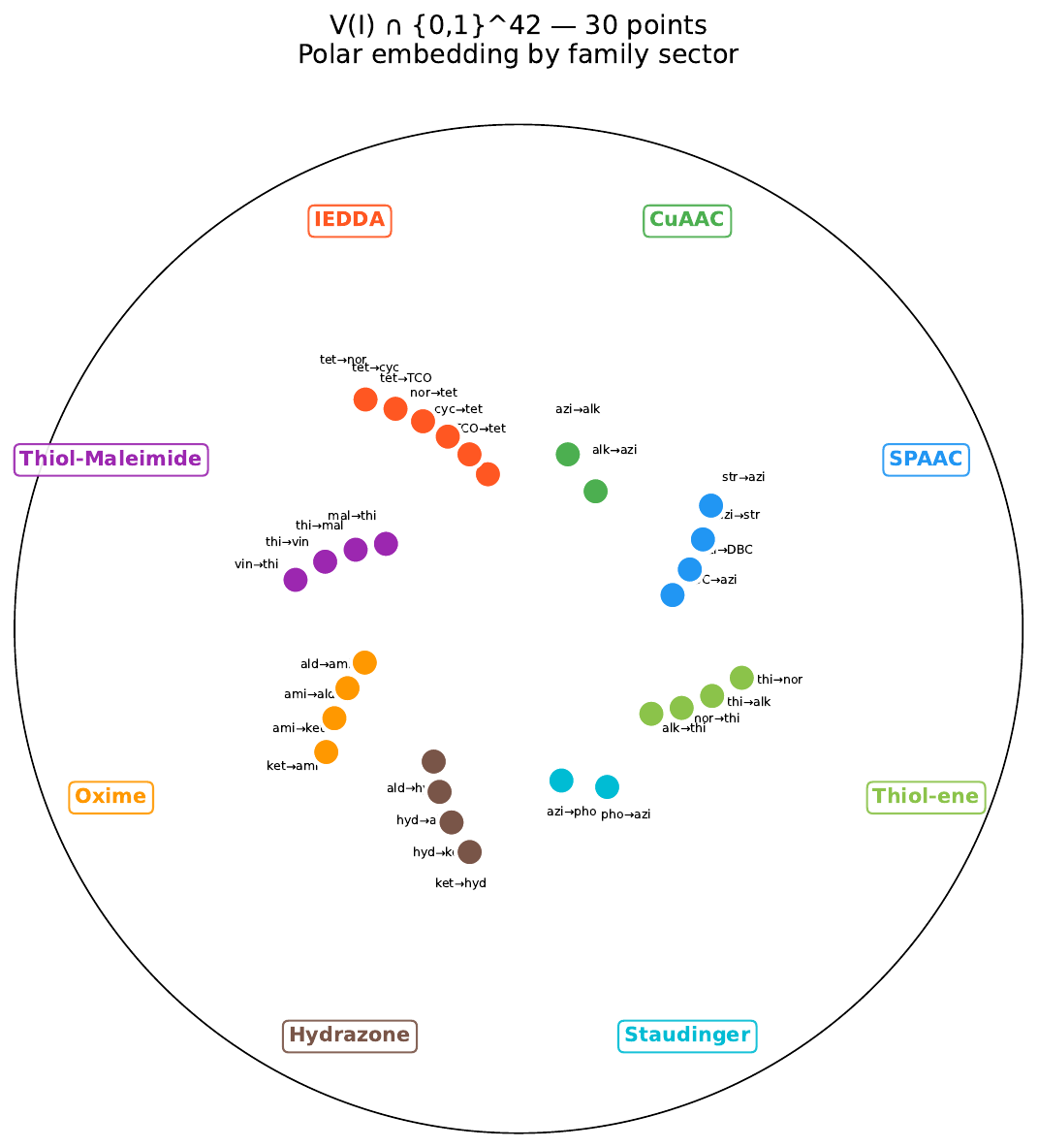}
\caption{Point cloud of $\Var(I) \cap \{0,1\}^{42}$, coloured
  by family.  Clusters correspond to fibres~$\Var(I_f)$.}
\label{fig:points}
\end{figure}

\paragraph{The assembly polytope.}
Taking the convex hull $\conv(\Var(I))$ of the 30~lattice points
and projecting onto the first three principal components yields
the polytope shown in \Cref{fig:segre} and~\cref{fig:polytope}.  Despite living
in $\RR^{42}$, the polytope admits a faithful 3D projection because
the effective dimension is controlled by the rank of the design
matrix ($\rank A = 28$).  The polytope's vertices are individual
assemblies, its edges connect assemblies differing in a single
handle choice, and its facets correspond to algebraic constraints
becoming tight.  This convex-geometric perspective connects
the assembly problem to the theory of lattice polytopes
and normal fans studied in toric geometry~\cite{Sturmfels96}.

\begin{figure}[h!]
\centering
\includegraphics[width=0.88\textwidth,trim=0 442 0 0,clip]{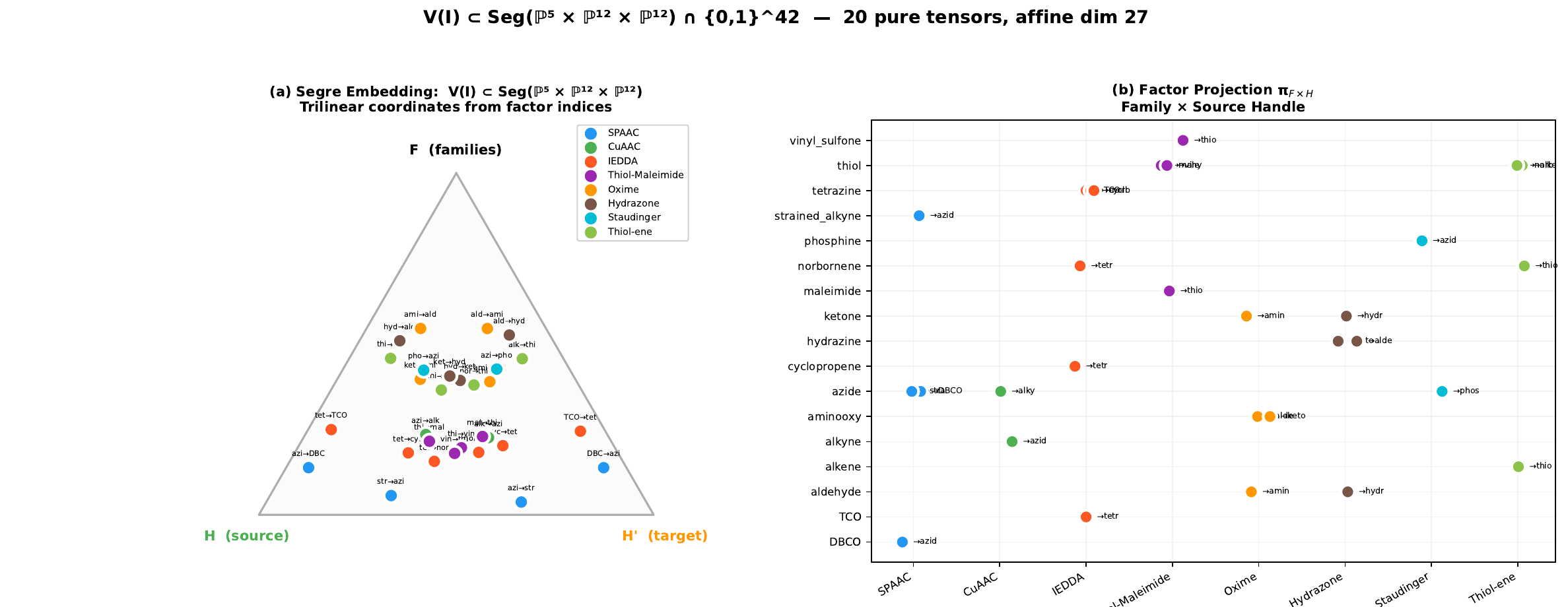}
\caption{Segre embedding and factor projections of $\Var(I)$.
  (a)~Trilinear coordinates from family, source-handle, and
  target-handle indices.
  (b)~Projection $\pi_{F \times H}$ onto the family~$\times$~source-handle plane.}
\label{fig:segre}
\end{figure}

\begin{figure}[h!]
\centering
\includegraphics[width=0.88\textwidth,trim=0 0 0 442,clip]{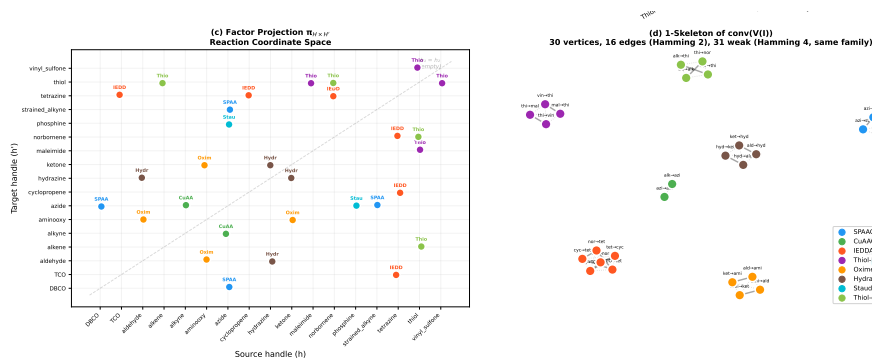}
\caption{Polytope geometry of $\conv(\Var(I)) \subset \RR^{42}$.
  (c)~Projection $\pi_{H \times H'}$ onto the reaction-coordinate plane.
  (d)~1-skeleton of $\conv(\Var(I))$: 30 vertices (assemblies) and
  16 edges (single-handle differences).}
\label{fig:polytope}
\end{figure}

\paragraph{The compatibility heatmap.}
\Cref{fig:projection} provides a complementary, discrete
view by projecting $\Var(I)$ onto the family~$\times$~handle-pair
grid.  Each cell $(f, (h, h'))$ is shaded if the triple
$(f, h, h')$ is feasible.  The dominant feature is the
block-diagonal structure imposed by $K_{\mathrm{compat}}$:
each family activates a small, contiguous block of handle pairs,
and most of the $8 \times 17^2 = 2312$ cells are empty.
The off-diagonal entries corresponding to shared handles are
clearly visible: for instance, azide columns are active in both
the SPAAC and CuAAC rows (and the Staudinger row), while
norbornene appears in both IEDDA and Thiol-ene rows.  These
shared columns are the visual fingerprint of the non-diagnostic
handles identified by the elimination ideal in
\Cref{thm:diagnostic}.

\begin{figure}[h!]
\centering
\includegraphics[width=0.92\textwidth]{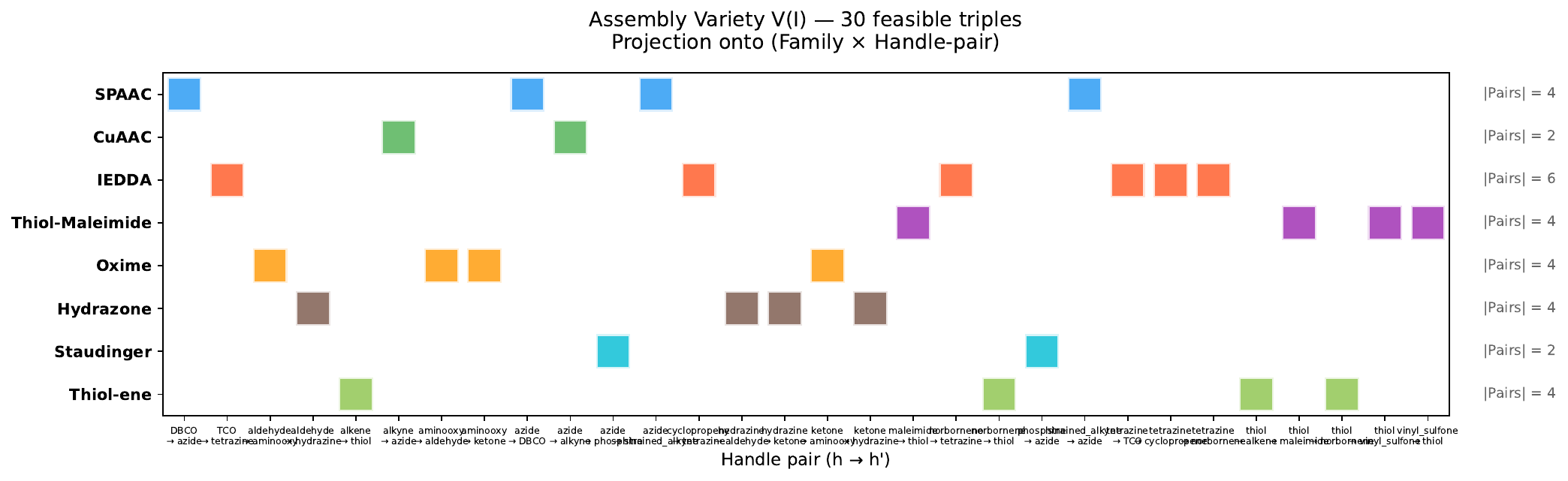}
\caption{Family~$\times$~handle-pair heatmap of~$\Var(I)$.
  Block-diagonal structure from $K_{\mathrm{compat}}$;
  shared-handle columns are the non-diagnostic entries.}
\label{fig:projection}
\end{figure}

\paragraph{The orthogonality graph and sequence.}
\Cref{fig:orthogonality} summarises the multi-step theory
of \cref{sec:multistep} in two panels.  The left panel shows the family-level handle-sharing graph
on 8~vertices and 6~edges; the four bottleneck corridors
are visually prominent.  The azide triangle
(SPAAC--CuAAC--Staudinger) forms a 3-clique, meaning at most one
of these three families can appear in any protocol.  The remaining
three edges (Oxime--Hydrazone, Thiol-Maleimide--Thiol-ene,
IEDDA--Thiol-ene) each eliminate one pairwise combination.
The right panel plots the orthogonality sequence
$(N_k)$ for $k = 1, \ldots, 8$.  The non-monotone
behaviour is striking: the count rises sharply from
$k = 1$ (30~single-step triples) to $k = 3$
(1152~three-step protocols), then plunges to 1024 at $k = 4$
before vanishing entirely at $k = 5$.  The peak at $k = 3$
reflects a combinatorial sweet spot where three pairwise-orthogonal
families can be assembled with maximal handle choice, while the
cliff at $k = 5$ is the sharp boundary imposed by
$\omega(G_\perp) = 4$.

\begin{figure}[h!]
\centering
\includegraphics[width=0.92\textwidth]{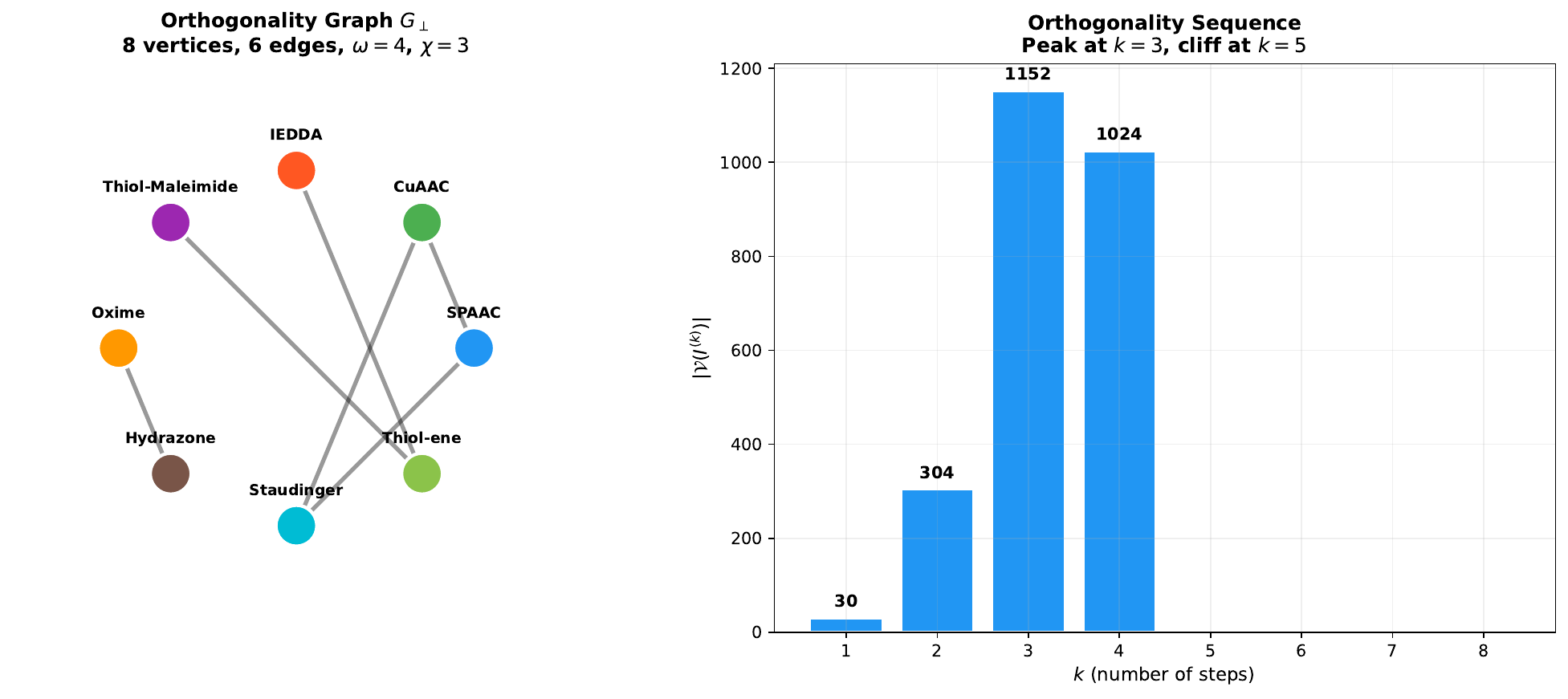}
\caption{Left: the family-level handle-sharing graph
  (8~vertices, 6~edges, 4~bottleneck corridors).
  Right: orthogonality sequence $(N_k)$
  peaking at $k = 3$ and vanishing at $k = 5$.}
\label{fig:orthogonality}
\end{figure}

\section{Discussion}\label{sec:discussion}

We have introduced an algebraic framework for compatibility-constrained
combinatorial assembly problems, encoding the design space as a Boolean
variety in a polynomial ring and reducing structural questions to
standard operations in commutative algebra.  The framework yields four
main results: the product-of-fields structure of~$R/I$
(\cref{sec:ring}), the elimination-theoretic characterisation of
handle diagnosticity (\cref{sec:elimination}), the two-generator
toric ideal with $\mathrm{MLdeg}(M) = 1$ (\cref{sec:statistics}),
and the $\omega(G_\perp) = 4$ bound with necessary and sufficient
conditions for its improvement (\cref{sec:multistep}).

\subsection{Limitations}

The framework encodes combinatorial feasibility on a Boolean
variety.  Continuous parameters (in the chemical setting:
molecular weight, lipophilicity, linker length) are not
represented; their incorporation would require extension to
semi-algebraic sets or real algebraic geometry
(\cref{sec:directions}).

Gr\"obner basis computation over~$\QQ$ is expensive for large
variable counts.  Working over $\mathbb{F}_2$ is algebraically
natural for Boolean ideals but has limited software support.
The fibre decomposition mitigates the computational cost by
reducing each computation to $2|H|$ variables
(\Cref{rem:fibre-scalability}).

\subsection{Research directions}\label{sec:directions}

We identify two open problems extending the present framework.

\subsubsection{Matroid structure of the feasible triples}

The set of 30~feasible triples, viewed as a subset of
$F \times H \times H'$, has a natural matroid-theoretic
interpretation.  Each triple $(f, h, h')$ can be identified
with a partial transversal of the bipartite graph between
families and handle pairs.  If the set of feasible triples
forms a \emph{transversal matroid} (in the sense of Edmonds--Fulkerson;
see~\cite{Oxley11}, Chapter~1), then the matroid's independent
sets correspond to simultaneously achievable assemblies in a
multi-target design, and the matroid rank equals the maximum
number of targets that can be addressed in a single protocol.

We have verified computationally that the feasible triples
satisfy the exchange axiom for subsets of size up to~4,
but a complete proof requires showing the augmentation
property for all subset sizes.  A positive answer would
connect the assembly problem to the extensive theory of
matroid intersection and matroid polytopes, opening
the door to polynomial-time algorithms for optimal
multi-target assembly selection.

\subsubsection{Continuous extension to semi-algebraic sets}

The current framework operates on the Boolean hypercube
$\{0,1\}^n$.  Incorporating continuous physicochemical properties
(molecular weight, PEG spacer length, lipophilicity) requires
moving to a semi-algebraic set
$S = \Var(I) \cap \{g_1 \geq 0, \ldots, g_s \geq 0\}$
where the~$g_j$ are polynomials encoding property windows.
The Positivstellensatz of Stengle (\cite{Stengle74};
see also~\cite{BCR98}, \S4.4) provides a certificate for
infeasibility of such systems, and Lasserre-type
sum-of-squares relaxations~\cite{Lasserre01} (cf.~also the toric
perspective of~\cite{Sturmfels96}) could yield practical algorithms
for optimising over~$S$.  The main challenge is that the Boolean
structure $J_{\mathrm{bool}}$ already forces~$R/I$ to be
zero-dimensional; adding continuous variables would create a
mixed discrete-continuous variety whose geometry differs
qualitatively from the purely Boolean case.

\section{Conclusions}\label{sec:conclusions}

\subsection{Algebraic contributions}

This paper demonstrates that compatibility-constrained combinatorial
assembly problems admit a unified algebraic treatment through Boolean
varieties and ideal theory.  The principal contributions are:

\begin{enumerate}[label=(\roman*),nosep]
\item The assembly ideal $I = J_{\mathrm{bool}} + J_{\mathrm{sel}} +
  K_{\mathrm{compat}}$ is radical, and the quotient $R/I$ decomposes
  as a product of copies of the ground field~$\kk$---one per feasible
  triple.  This product-of-fields structure equips the variety with a
  layered filtration whose successive quotients are controlled by
  explicit ideal quotients.

\item Elimination ideals detect handle diagnosticity: membership of
  $h(1 - f_0)$ in $I \cap \kk[F,H]$ certifies that a handle~$h$
  belongs to a unique family.  This algebraic criterion replaces
  ad-hoc enumeration with a certificate that is both machine-checkable
  and structurally informative.

\item For the 8-family instance, the toric ideal of the
  log-linear model on~$\Var(I)$ is generated by exactly two
  binomials, and the ML degree equals~1.  The maximum likelihood
  estimator is therefore a rational function of the data, placing
  the model in the simplest possible class from the perspective of
  algebraic statistics.

\item For the 8-family instance, the multi-step ideal $I^{(k)}$
  encodes simultaneous orthogonality constraints.  The clique number
  $\omega(G_\perp) = 4$ of the orthogonality graph provides an exact
  upper bound on the number of mutually compatible reactions, and the
  obstruction is concentrated in four bottleneck corridors of the
  cross-reactivity graph.
\end{enumerate}

Together, these results show that standard tools from commutative
algebra---Gr\"obner bases, elimination, toric ideals, and graph
theory---can extract structural information from combinatorial
design spaces that procedural enumeration alone cannot reveal.

\subsection{What the algebra reveals about bioorthogonal click chemistry}

Applied to the landscape of bioorthogonal click chemistry, the
algebraic framework yields a number of concrete insights about
the practical design of multiplexed ligation protocols.
A first indication of the framework's reach is the sheer sparsity of
the design space: of the $8 \times 17^2 = 2312$ conceivable
family--handle--handle combinations, only 30 (1.3\%) are chemically
feasible.  The algebra makes this precise and traces the 98.7\%
exclusion to explicit generators of the assembly ideal.

Of the 17~reactive functional groups (handles) catalogued across the
eight established reaction families, 12 turn out to be
\emph{diagnostic}: each one participates in exactly one reaction
family, so observing the handle immediately identifies the underlying
chemistry.  The five non-diagnostic handles---notably azide (shared by
CuAAC, SPAAC, and Staudinger) and norbornene (shared by IEDDA and
Thiol-ene)---are the sole sources of cross-reactivity in the system.
A practical consequence is that once a chemist has committed to a
particular reaction family, every compatible handle pair within that
family is equally available: the choice of source handle does not
constrain the target handle.  Design constraints only arise when
multiple families are combined in a single protocol, so optimisation
\emph{within} a family---for instance, selecting the least toxic
azide variant for a SPAAC step---is unconstrained by the algebra.

This cross-reactivity is tightly structured.  The entire algebraic
redundancy of the system is captured by just two binomial relations,
both reflecting the same chemical phenomenon: aldehyde and ketone
are interchangeable as carbonyl handles within the oxime and
hydrazone families.  No other hidden symmetries exist.
A related statistical consequence is that the maximum likelihood
estimator for assembly frequencies has a closed-form rational
expression (ML degree~$= 1$): if one measures how often each
assembly appears in a combinatorial library or high-throughput screen,
the best-fit model can be written down explicitly---no iterative
fitting is required, and there is no risk of converging to a
spurious local optimum.

The most consequential finding concerns multiplexing.  Although the
eight reaction families might suggest that up to eight orthogonal
reactions could run simultaneously, the algebra shows the true
ceiling is four.  The obstruction traces to four specific
cross-reactivity corridors in the handle-sharing graph---at the
azide, carbonyl, thiol, and norbornene positions---where distinct
families compete for the same reactive group.  Protocols of three simultaneous reactions enjoy the greatest
design freedom ($N_3 = 1152$ unordered protocols, distributed
over 23~distinct family subsets); at four reactions, the design
space contracts to $N_4 = 1024$ protocols across just six family
subsets, and at five the combinatorial constraints become
infeasible altogether.
Moreover, every maximally multiplexed (four-reaction) protocol must
include both IEDDA and Thiol-Maleimide: these two families appear in
all six maximum independent sets because they share no handles with
any family that can itself participate in a maximum clique.  (IEDDA
shares norbornene with Thiol-ene, and Thiol-Maleimide shares thiol
with Thiol-ene, but Thiol-ene is excluded from every MIS by its
conflicts in the thiol and norbornene corridors.)  Leaving either
one out of a multiplexed design wastes a guaranteed conflict-free
slot.  The remaining two
slots are then filled by choosing one family from each side of the
azide corridor (SPAAC \emph{or} CuAAC, never both) and one from the
carbonyl corridor (Oxime \emph{or} Hydrazone, never both).

These conclusions---handle diagnosticity, within-family
handle freedom, carbonyl equivalence, the closed-form
MLE, the four-reaction ceiling, the IEDDA/Thiol-Maleimide
anchor, and the sharp cliff beyond---do not depend
on kinetic rates, solvent
conditions, or steric considerations.  They are purely combinatorial
consequences of which handles pair with which families, extracted by
the algebraic machinery developed in this paper.

Finally, the framework provides an explicit \emph{family extension
criterion} for evaluating candidate reactions before any synthesis is
undertaken.  \Cref{prop:omega5} reduces the question ``will
a new reaction family raise the multiplexing ceiling?'' to a
handle-disjointness check---supplemented by the absence of
third-family mediated cross-reactivity---against the six maximum
independent sets of the current orthogonality graph.  For families
introducing entirely novel handles, the check is automatic.  If
the new family's handles are already used by a family present in
every maximum independent set (in the current landscape: IEDDA or
Thiol-Maleimide), the ceiling cannot rise.  Among six emerging reactions surveyed, five
pass this test---most notably SuFEx, whose entirely novel handles
(sulfonyl fluoride and silyl ether) raise $\omega$ to~5
unconditionally and shift the orthogonality cliff from $k = 5$ to
$k = 6$, opening 2048 five-step protocols that are currently
infeasible.  The sole exception, isonitrile--tetrazine, fails because
it shares tetrazine with IEDDA.  This criterion gives experimentalists
a purely combinatorial screening tool: before investing in the
development of a new bioorthogonal handle pair, one can check in
advance whether it will actually expand the multiplexing capacity of
the toolkit.


\section*{Acknowledgments}
\textbf{Disclosure of AI assistance.}
This paper reports original research conceived, directed,
proved, and verified by the author.  The author formulated
the problem, designed the algebraic framework---the
family--handle pair structure, the assembly ideal $I$, the
choice of invariants (zero-dimensionality, radicality,
primary decomposition, toric reduction, ML degree, the
multi-step ideal $I^{(k)}$, the orthogonality graph
$G_\perp$, and the $\omega = 5$ criterion)---identified
bioorthogonal click chemistry as the target application,
selected the Kadcyla dual-click example, stated and proved
the theorems, and interpreted the biological and design-level
consequences.  An AI assistant (Claude, Anthropic, April~2026)
was used under the author's direction as a technical tool
for the following mechanical tasks: implementing Python
scripts from the author's specifications
(\texttt{encode.py}, \texttt{multistep.py},
\texttt{statistics.py}) to carry out numerical verification
of the author's algebraic claims, producing figures
following the author's designs, \LaTeX{} typesetting, and
copy-editing of prose drafted by the author.  All
mathematical content, proofs, numerical results, design
choices, and editorial decisions are the author's; the
author reviewed and approved every line of the final
manuscript and code, and assumes sole responsibility for
the integrity and accuracy of the work.  The complete
codebase is available at
\url{https://github.com/graonet/AlgeClick} (currently
private; to be made public upon acceptance).



\begin{thebibliography}{22}
\providecommand{\natexlab}[1]{#1}
\providecommand{\url}[1]{\texttt{#1}}
\expandafter\ifx\csname urlstyle\endcsname\relax
  \providecommand{\doi}[1]{doi: #1}\else
  \providecommand{\doi}{doi: \begingroup \urlstyle{rm}\Url}\fi

\bibitem[Barndorff-Nielsen(1978)]{BN78}
O.~E. Barndorff-Nielsen.
\newblock \emph{Information and Exponential Families in Statistical Theory}.
\newblock Wiley, 1978.

\bibitem[Bochnak et~al.(1998)Bochnak, Coste, and Roy]{BCR98}
J.~Bochnak, M.~Coste, and M.-F. Roy.
\newblock \emph{Real Algebraic Geometry}, volume~36 of \emph{Ergebnisse der
  Mathematik und ihrer Grenzgebiete}.
\newblock Springer, 1998.

\bibitem[Cox et~al.(2015)Cox, Little, and O'Shea]{CLO15}
D.~Cox, J.~Little, and D.~O'Shea.
\newblock \emph{Ideals, Varieties, and Algorithms}.
\newblock Springer, 4th edition, 2015.

\bibitem[Csisz{\'a}r(1975)]{Csiszar75}
I.~Csisz{\'a}r.
\newblock {$I$}-divergence geometry of probability distributions and
  minimization problems.
\newblock \emph{Ann.\ Probab.}, 3\penalty0 (1):\penalty0 146--158, 1975.

\bibitem[Devaraj and Weissleder(2011)]{Devaraj11}
N.~K. Devaraj and R.~Weissleder.
\newblock Biomedical applications of tetrazine cycloadditions.
\newblock \emph{Acc.\ Chem.\ Res.}, 44\penalty0 (9):\penalty0 816--827, 2011.

\bibitem[Dickenstein(2016)]{Dickenstein16}
A.~Dickenstein.
\newblock Biochemical reaction networks: an invitation for algebraic geometers.
\newblock In \emph{Algebraic and Geometric Methods in Discrete Mathematics},
  volume 656 of \emph{Contemp.\ Math.}, pages 65--83. American Mathematical
  Society, 2016.

\bibitem[Drton et~al.(2009)Drton, Sturmfels, and Sullivant]{DSS09}
M.~Drton, B.~Sturmfels, and S.~Sullivant.
\newblock \emph{Lectures on Algebraic Statistics}.
\newblock Oberwolfach Seminars. Birkh\"auser, 2009.

\bibitem[Eisenbud(1995)]{Eisenbud95}
D.~Eisenbud.
\newblock \emph{Commutative Algebra with a View Toward Algebraic Geometry},
  volume 150 of \emph{Graduate Texts in Mathematics}.
\newblock Springer, 1995.

\bibitem[Feinberg(2019)]{Feinberg19}
M.~Feinberg.
\newblock \emph{Foundations of Chemical Reaction Network Theory}.
\newblock Springer, 2019.

\bibitem[Feliu and Shiu(2025)]{FeliuShiu25}
E.~Feliu and A.~Shiu.
\newblock From chemical reaction networks to algebraic and polyhedral
  geometry---and back again.
\newblock In \emph{Varieties, Polyhedra, Computation}, EMS Series of Congress
  Reports. European Mathematical Society, 2025.
\newblock arXiv:2501.06354.

\bibitem[Lang and Chin(2014)]{Lang14}
K.~Lang and J.~W. Chin.
\newblock Bioorthogonal reactions for labeling living systems.
\newblock \emph{ACS Chem.\ Biol.}, 9\penalty0 (1):\penalty0 16--20, 2014.

\bibitem[Lasserre(2001)]{Lasserre01}
J.~B. Lasserre.
\newblock Global optimization with polynomials and the problem of moments.
\newblock \emph{SIAM J.\ Optim.}, 11\penalty0 (3):\penalty0 796--817, 2001.

\bibitem[McKay and Finn(2014)]{McKay14}
C.~S. McKay and M.~G. Finn.
\newblock Click chemistry in complex mixtures: bioorthogonal bioconjugation.
\newblock \emph{Chem.\ Biol.}, 21\penalty0 (9):\penalty0 1075--1101, 2014.

\bibitem[Meurer et~al.(2017)]{SymPy}
A.~Meurer et~al.
\newblock {SymPy}: symbolic mathematics in {Python}.
\newblock \emph{PeerJ Comput.\ Sci.}, 3:\penalty0 e103, 2017.

\bibitem[Oxley(2011)]{Oxley11}
J.~Oxley.
\newblock \emph{Matroid Theory}, volume~21 of \emph{Oxford Graduate Texts in
  Mathematics}.
\newblock Oxford University Press, 2nd edition, 2011.

\bibitem[Pachter and Sturmfels(2005)]{PS05}
L.~Pachter and B.~Sturmfels.
\newblock \emph{Algebraic Statistics for Computational Biology}.
\newblock Cambridge University Press, 2005.

\bibitem[Rostovtsev et~al.(2002)Rostovtsev, Green, Fokin, and
  Sharpless]{Rostovtsev02}
V.~V. Rostovtsev, L.~G. Green, V.~V. Fokin, and K.~B. Sharpless.
\newblock A stepwise {Huisgen} cycloaddition process: copper({I})-catalyzed
  regioselective ``ligation'' of azides and terminal alkynes.
\newblock \emph{Angew.\ Chem.\ Int.\ Ed.}, 41\penalty0 (14):\penalty0
  2596--2599, 2002.

\bibitem[Sletten and Bertozzi(2009)]{Sletten09}
E.~M. Sletten and C.~R. Bertozzi.
\newblock Bioorthogonal chemistry: fishing for selectivity in a sea of
  functionality.
\newblock \emph{Angew.\ Chem.\ Int.\ Ed.}, 48\penalty0 (38):\penalty0
  6974--6998, 2009.

\bibitem[Stengle(1974)]{Stengle74}
G.~Stengle.
\newblock A {Nullstellensatz} and a {Positivstellensatz} in semialgebraic
  geometry.
\newblock \emph{Math.\ Ann.}, 207:\penalty0 87--97, 1974.

\bibitem[Sturmfels(1996)]{Sturmfels96}
B.~Sturmfels.
\newblock \emph{{Gr\"obner} Bases and Convex Polytopes}.
\newblock AMS University Lecture Series. American Mathematical Society, 1996.

\bibitem[Torn{\o}e et~al.(2002)Torn{\o}e, Christensen, and Meldal]{Tornoe02}
C.~W. Torn{\o}e, C.~Christensen, and M.~Meldal.
\newblock Peptidotriazoles on solid phase: [1,2,3]-triazoles by regiospecific
  copper({I})-catalyzed 1,3-dipolar cycloadditions of terminal alkynes to
  azides.
\newblock \emph{J.~Org.\ Chem.}, 67\penalty0 (9):\penalty0 3057--3064, 2002.

\bibitem[Wainwright and Jordan(2008)]{WainwrightJordan08}
M.~J. Wainwright and M.~I. Jordan.
\newblock Graphical models, exponential families, and variational inference.
\newblock \emph{Foundations and Trends in Machine Learning}, 1\penalty0
  (1--2):\penalty0 1--305, 2008.

\end{thebibliography}
\end{document}